\documentclass[12pt, twoside]{article}

\usepackage{amssymb,amsthm,amsthm,times}

\def\titlerunning#1{\gdef\titrun{#1}}
\makeatletter
\def\author#1{\gdef\autrun{\def\and{\unskip, }#1}\gdef\@author{#1}}
\def\address#1{{\def\and{\\\hspace*{18pt}}\renewcommand{\thefootnote}{}%
\footnote {#1}}%
\markboth{\autrun}{\titrun}}
\makeatother
\def\email#1{e-mail: #1}
\def\subjclass#1{{\renewcommand{\thefootnote}{}%
\footnote{\emph{Mathematics Subject Classification (2010):} #1}}}
\def\keywords#1{\par\medskip
\noindent\textbf{Keywords.} #1}

\newtheorem{theorem}{Theorem}
\newtheorem{lemma}[theorem]{Lemma}
\newtheorem{corollary}[theorem]{Corollary}

\newtheorem{proposition}[theorem]{Proposition}

\theoremstyle{definition}
\newtheorem{definition}[theorem]{Definition}

\newcommand{\A}{\mathrm A}
\newcommand{\B}{\mathrm B}
\newcommand{\C}{\mathrm C}
\newcommand{\D}{\mathrm D}
\newcommand{\K}{\mathrm K}
\newcommand{\V}{\mathrm V}
\newcommand{\Z}{\mathrm Z}

\newcommand{\AAA}{\mathrm {A_3G}}
\newcommand{\AAG}{\mathrm {A^2G}}
\newcommand{\Line}{\mathrm L}
\newcommand{\AG}{\mathrm {AG}}
\newcommand{\HC}{\mathrm {HC}}
\newcommand{\Pet}{F_{10}}
\newcommand{\Hea}{F_{14}}
\newcommand{\Tut}{F_{30}}

\newcommand{\vC}{\vec{\C}}
\newcommand{\vGa}{\vec{\Gamma}}

\renewcommand{\wr}{\mathop{\rm wr}}

\newcommand{\Cay}{\mathrm{Cay}}
\newcommand{\Cos}{\mathrm{Cos}}
\newcommand{\Aut}{\mathrm{Aut}}
\newcommand{\Out}{\mathrm{Out}}
\newcommand{\soc}{\mathrm{soc}}
\newcommand{\Alt}{\mathrm{Alt}}
\newcommand{\Sym}{\mathrm{Sym}}
\newcommand{\SL}{\mathrm{SL}}
\newcommand{\GL}{\mathrm{GL}}
\newcommand{\PSL}{\mathrm{PSL}}
\newcommand{\PGL}{\mathrm{PGL}}
\newcommand{\PGammaL}{\mathrm{P\Gamma L}}

\newcommand{\Merge}{\mathrm {MG}}
\newcommand{\Split}{\mathrm {SG}}

\begin{document}

\baselineskip=17pt

\titlerunning{Order of the vertex-stabiliser in $4$-valent arc-transitive graphs}
\title{Bounding the order of the vertex-stabiliser in $3$-valent vertex-transitive and $4$-valent arc-transitive graphs}
\author{
Primo\v{z} Poto\v{c}nik \and
Pablo Spiga \and
Gabriel Verret
}
\date{}
\maketitle
\address{P.~Poto\v{c}nik: Institute of Mathematics, Physics, and
  Mechanics, Jadranska 19, 1000 Ljubljana, Slovenia; \email{primoz.potocnik@fmf.uni-lj.si}
\and
P.~Spiga (corresponding author):  School of Mathematics and Statistics,
The University of Western Australia,
Crawley, WA 6009, Australia; \email{spiga@maths.uwa.edu.au}
\and
G.~Verret: Institute of Mathematics, Physics, and
  Mechanics, 
Jadranska 19, 1000 Ljubljana, Slovenia;
\email{gabriel.verret@fmf.uni-lj.si}}



\begin{abstract}
The main result of this paper is that, if $\Gamma$ is a connected 4-valent $G$-arc-transitive graph and $v$ is a vertex of $\Gamma$, then either $\Gamma$ is one of a well understood infinite family of graphs, or $|G_v|\leq 2^43^6$ or $2|G_v|\log_2(|G_v|/2)\leq |\V\Gamma|$ and that this last bound is tight. As a corollary, we get a similar result for $3$-valent vertex-transitive graphs. 

\end{abstract}


\section{Introduction}\label{intro}
The question ``how symmetric is a certain mathematical object?" has a venerable history. In general, this question is rather vague but a natural starting point is to consider the order of the automorphism group of the object. This is especially true in the case of finite objects. Of course, larger objects have the potential to admit much larger automorphism groups hence it may be more fruitful to compare the size of the object with the order of its automorphism group. This is the point of view we adopt in this paper. The objects we consider are finite $3$-valent vertex-transitive and $4$-valent arc-transitive graphs. The main result is a striking dichotomy between a well understood family of exceptional graphs, each having a very large automorphisms group and the rest of the graphs with comparatively small automorphism groups.

We first fix some terminology and mention some background results. Throughout this paper, all graphs considered will be finite, except in Section~\ref{Amalgams}. A graph $\Gamma$ is said to be $G$-\emph{vertex-transitive} if $G$ is a subgroup of $\Aut(\Gamma)$ acting transitively on the vertex-set $\V\Gamma$ of $\Gamma$. Similarly, $\Gamma$ is said to be $G$-\emph{arc-transitive} if $G$ acts transitively on the arcs of $\Gamma$ (that is, on the ordered pairs of adjacent vertices of $\Gamma$). When $G=\Aut(\Gamma)$, the prefix $G$ in the above notation is sometimes omitted. 

A celebrated theorem of Tutte~\cite{Tutte,Tutte2} shows that, if $\Gamma$ is a connected 3-valent $G$-arc-transitive graph, then the stabiliser of a vertex in $G$ has order at most $48$. It is very natural to try to relax the hypothesis of this remarkable theorem by considering valencies greater than 3. In this vein, it can be deduced from the work of Trofimov~\cite{Trof,Trof2} and  Weiss~\cite{Wep} that, if $p$ is a prime, then there exists a constant $c_p$ depending only on $p$ such that, if $\Gamma$ is a connected $p$-valent $G$-arc-transitive graph, then the stabiliser of a vertex in $G$ has order at most $c_p$, generalising the result of Tutte. The situation is quite different when the valency is not a prime, as the next example will show.

We define a family of 4-valent graphs which we will denote $\C(r,s)$. These were studied in detail by Gardiner, Praeger and Xu~\cite{GarPra2,PraegerXu}. We give a definition which is slightly different, but equivalent to the definition used in~\cite{GarPra2}. Furthermore, as we will be mainly interested in $4$-valent graphs, we simply denote by $\C(r,s)$ the graphs denoted by $\C(2,r,s)$ in~\cite{GarPra2}. 
Let $r$ and $s$ be positive integers with $r\geq 3$ and $1\leq s\leq r-1$. Let $\C(r,1)$ be the lexicographic product $\C_r[\overline{\K_2}]$ of a cycle of length $r$ and an edgeless graph on $2$ vertices. In other words, $\V(\C(r,1))=\mathbb{Z}_r\times\mathbb{Z}_2$ with $(u,i)$ being adjacent to $(v,j)$ if and only if $|v-u|=1$. Further, for $s\ge2$, let $\C(r,s)$ be the graph with vertices being the $(s-1)$-paths of $\C(r,1)$ containing at most one vertex from $\{(y,0),(y,1)\}$ for each $y\in\mathbb{Z}_r$ and with two such $(s-1)$-paths being adjacent in $\C(r,s)$ if and only if their intersection is an $(s-2)$-path in $\C(r,1)$. Clearly, $\C(r,s)$ is a connected $4$-valent graph with $r2^s$ vertices. 

There is an obvious action of the wreath product $H=\C_2\wr \D_r$ on $\C_r[\overline{\K}_2]=\C(r,1)$ which induces an arc-transitive action on $\C(r,s)$ for $s\leq r-1$. Note that $|H|=2r2^r$ and hence the order of the stabiliser of a vertex of $\C(r,s)$ in $H$ is $2^{r-s+1}$, which is unbounded. Moreover, if we fix $s$, then  the order of the stabiliser of a vertex of $\C(r,s)$ grows exponentially with $r$ and hence exponentially  with the number of vertices of $\C(r,s)$. 

It has long been suspected that the graphs $\C(r,s)$ are rather exceptional in this respect. For example, Xu asked whether every $4$-valent $G$-arc-transitive graph with $|G_v| > 2^43^6$ is isomorphic to some $\C(r,s)$ (see \cite[Problem 17]{XU}). The answer is negative, as can be seen with a construction of Gardiner and Praeger. 

For each $s\geq 3$, they construct an infinite family of $4$-valent $G$-arc-transitive graphs $\Gamma$ (denoted by
$\C^{\pm 1}(3,s,s)$ in~\cite[Definition 2.2]{GarPra2}) with $2^{s+1}=|G_v|\leq
|V\Gamma|^{\log_3(2)}$. Another example is constructed by Conder and Walker \cite{ConderWalker}, who construct an infinite family of $4$-valent $G$-arc-transitive non-Cayley graphs $\Gamma$ such that $G\cong\Sym(n)$ for some $n$. While $|G_v|$ is unbounded in both these examples, it grows rather mildly with $|\V\Gamma|$ compared to the exponential growth exhibited by the graphs $\C(r,s)$. In fact, our main result is that, excluding the graphs $\C(r,s)$, $|G_v|$ is indeed bounded above by a sub-linear function of $|\V\Gamma|$. Before we can state Theorem~\ref{thm:main} in its full generality, we need to define the following very important concept.

\begin{definition} \label{def:locally}

Let $P$ be a permutation group, let $\Gamma$ be a connected $G$-vertex-transitive graph and let $v$ be a vertex of $\Gamma$. We denote by $G_v^{\Gamma(v)}$ the permutation group induced by the stabiliser $G_v$ of the vertex $v\in \V\Gamma$ on the neighbourhood $\Gamma(v)$. If $G_v^{\Gamma(v)}$ is permutation isomorphic to $P$, then we say that $(\Gamma,G)$ is \emph{locally-$P$}. 
\end{definition}

If $\Gamma$ has valency $k$, then the permutation group $G_v^{\Gamma(v)}$ has degree $k$ and, up to permutation isomorphism, does not depend on the choice of $v$. It is an elementary observation that $G_v^{\Gamma(v)}$ is transitive if and only if $\Gamma$ is $G$-arc-transitive and it is regular if and only if  $\Gamma$ is $G$-arc-regular, in which case $|G_v|=k$.

Let $\Gamma$ be a connected 4-valent $G$-arc-transitive graph and let $v\in \V\Gamma$. It follows from the work of Gardiner~\cite{Gardiner} that, if $G_v^{\Gamma(v)}$ is $2$-transitive, then $|G_v|\leq 2^43^6$. Up to permutation isomorphism, there is only one transitive permutation group of degree $4$ that is neither regular nor $2$-transitive, namely $\D_4$, the dihedral group of order $8$ in its action on $4$ points. By the elementary observation above together with the work of Gardiner, we obtain that if $|G_v|> 2^43^6$, then $(\Gamma,G)$ is locally-$\D_4$. This  shows that the hypothesis of our main result is not restrictive.

\begin{theorem}
\label{thm:main}
Let $(\Gamma,G)$ be locally-$\D_4$. Then one of the following holds:
\begin{description}
\item[$(A)$] $\Gamma\cong \C(r,s)$ for some $r\geq 3$, $1\leq s\leq \frac{r}{2}$;
\item[$(B)$] $(\Gamma,G)$ is one of the pairs in Table~\ref{tb:soluble} or Table~\ref{tb:nsoluble};
\item[$(C)$]  $|\V\Gamma| \ge 2|G_v|\log_2(|G_v|/2)$.
\end{description}
Moreover, if~$(C)$ holds with equality and $\Gamma$ is not as in~$(A)$, then $(\Gamma,G)$ is one of $(\Gamma_t^+,G_t^+)$ or $(\Gamma_t^-,G_t^-)$ for some $t\geq 2$.
\end{theorem}

Table~\ref{tb:soluble} and Table~\ref{tb:nsoluble} as well as the definition of the pairs $(\Gamma_t^+,G_t^+)$ and $(\Gamma_t^-,G_t^-)$ can be found in Section~\ref{subsec}. If $(\Gamma,G)$ is one of the pairs in  Table~\ref{tb:soluble} or Table~\ref{tb:nsoluble}, then $|G_v|\leq 512<2^43^6$. Hence, Theorem~\ref{thm:main} together with the work of Gardiner has the following corollary.

\begin{corollary}
\label{cor:main}
Let $\Gamma$ be a connected $4$-valent $G$-arc-transitive graph. Then one of the following holds:
\begin{description}
\item[$(A)$] $\Gamma\cong \C(r,s)$ for some $r\geq 3$, $1\leq s\leq \frac{r}{2}$;
\item[$(B)$] $|G_v|\leq 2^43^6$;
\item[$(C)$]  $|\V\Gamma| \ge 2|G_v|\log_2(|G_v|/2)$.
\end{description}
Moreover, if~$(C)$ holds with equality and $\Gamma$ is not as in~$(A)$, then $(\Gamma,G)$ is one of $(\Gamma_t^+,G_t^+)$ or $(\Gamma_t^-,G_t^-)$ for some $t\geq 2$.
\end{corollary}

For each locally-$\D_4$ pair $(\Gamma,G)$, there is a natural way to construct a $3$-valent $G$-vertex-transitive graph $\Split(\Gamma)$ with $|\V(\Split(\Gamma))|=2|\V\Gamma|$. In some appropriate sense, this construction is reversible. More details can be found in Section~\ref{ss:cubic}, where we prove the following:

\begin{corollary}\label{Cubic}
Let $\Gamma$ be a connected $3$-valent $G$-vertex-transitive graph. Then one of the following holds:
\begin{description}
\item[$(A)$] $\Gamma=\Split(\Gamma')$ where either $\Gamma'$ appears in Table~\ref{tb:soluble} or in Table~\ref{tb:nsoluble} or $\Gamma'\cong \C(r,s)$ for some $r\geq 3$, $1\leq s\leq \frac{r}{2}$;
\item[$(B)$] $\Gamma$ is $G$-arc-transitive and $|G_v|\leq 48$;
\item[$(C)$]  $|\V\Gamma|\geq 8|G_v|\log_2|G_v|$.
\end{description}
\end{corollary}

\subsection{Structure of the paper and sketch of the proof of Theorem~\ref{thm:main}}\label{Struc} 
Let $(\Gamma,G)$ be locally-$\D_4$. 
We prove Theorem~\ref{thm:main} by considering the action of a minimal normal subgroup $N$ of $G$ on $\V\Gamma$. The {\em quotient graph} $\Gamma/N$ is the graph whose vertices are the $N$-orbits on $\V\Gamma$ with  two such $N$-orbits $v^N$ and $u^N$ adjacent whenever there is a pair of vertices $v'\in v^N$ and $u'\in u^N$ that are adjacent in $\Gamma$. Observe that $G/N$ acts on $\Gamma/N$ arc-transitively, and that the valency of $\Gamma/N$ is either $0$ (when $N$ is transitive on $\V\Gamma$), $1$ (when $N$ has $2$ orbits on $\V\Gamma$), $2$ (when $\Gamma/N$ is a cycle) or $4$. In the latter case, $G/N$ acts faithfully on $\V\Gamma$ and hence $(\Gamma/N,G/N)$ is locally-$\D_4$ with the vertex-stabiliser $(G/N)_{v^N}=G_vN/N$ in $G/N$ isomorphic to $G_v$. Therefore, this will allow the use of an inductive argument when $\Gamma/N$ has valency $4$.

In Section~\ref{sec:abelian} we study the case when $N$ is abelian. Namely, in Section~\ref{sec:basic}, we consider the case when $\Gamma/N$ has valency at most $2$. Next, if $\Gamma/N$ has valency $4$, then by induction we may assume that
$\Gamma/N$ is one of the graphs in $(A)$ or $(B)$ of Theorem~\ref{thm:main}. The case when $\Gamma/N$ is as in~$(A)$ is dealt with in Section~\ref{sec:nc}. Finally, the case when $\Gamma/N$ is as in~$(B)$  requires a few computations which are carried out in the proof of Lemma~\ref{thm:exp2}.

In Section~\ref{eaAS} we study the case when $N$ is non-abelian. The main ingredient in this section is  a result on the order of elementary abelian subgroups in simple groups (Theorem~\ref{theo:key}). The proof of Theorem~\ref{theo:key} is very technical, uses the Classification of Finite Simple Groups and is delayed until Section~\ref{TechnicalProof}.

The proof of Theorem~\ref{thm:main} is in Section~\ref{MainProof} and consists in collecting all the preceding partial results.

Section~\ref{OtherStuff} consists of applications of our main result and additional remarks. In Section~\ref{ss:cubic}, we show that the problem of bounding the order of the vertex-stabiliser of a $3$-valent vertex-transitive graph is equivalent to the problem of bounding it for $4$-valent arc-transitive graphs and prove Corollary~\ref{Cubic}. Finally, in Section~\ref{Amalgams}, we explain how to rephrase our results in a purely group theoretical language and how they can be interpreted as bounds on the indices of some normal subgroups in some infinite groups.

\textbf{Remark:} Part of the proof of Theorem~\ref{thm:main} relies on  the Classification of Finite Simple Groups. Using methods similar to the techniques developed in~\cite{VSINHAT}, it is possible to prove  (without using the Classification of Finite Simple Groups) the following much weaker version of Theorem~\ref{thm:main}.

\begin{theorem}
\label{thm:weakmain}
Let $(\Gamma,G)$ be locally-$\D_4$. Then either $|G_v|\leq2 |\V\Gamma|^3$ or $\Gamma\cong \C(r,s)$ for some $1\leq s\leq r-1$.
\end{theorem}

\section{Exceptions in Theorem~\ref{thm:main}}\label{subsec}
In this section, we describe the exceptional graphs in Theorem~\ref{thm:main} and we state some preliminary results regarding these families that are needed in the rest of the paper. The graphs $\C(r,s)$ were
introduced in Section~\ref{intro}.

\subsection{The graphs $\Gamma_t^+$ and $\Gamma_t^-$. }
\label{ss:Gamma}

In this section, we describe the pairs $(\Gamma_t^+,G_t^+)$ and $(\Gamma_t^-,G_t^-)$ mentioned in Theorem~\ref{thm:main} (these are the pairs attaining the bound $|\V\Gamma| = 2|G_v| \log_2(|G_v|/2)$ in part~$(C)$ of Theorem~\ref{thm:main}).
These graphs are studied in detail in \cite{gamma} and what follows is a brief overview of some facts relevant for the topic of this paper.
The graphs $\Gamma_t^\pm$ are defined as \emph{coset graphs} of certain groups $G_t^\pm$. The coset graph $\Cos(G,H,a)$ on a group $G$ relative to a subgroup $H\le G$ and an element $a\in G$ is defined as the graph with vertex set 
 the set of right cosets $G/H = \{Hg \mid g \in G\}$ and with edge set the set $ \{ \{Hg, Hag\} \mid g \in G\} $.

Let us start by considering the group $E_t$ with the following presentation
\begin{eqnarray}\label{eq:E}\nonumber
E_t=\langle x_0,\ldots,x_{2t-1},z&\mid&
x_i^2=z^2=[x_i,z]=1 \textrm{ for }0\leq i\leq 2t-1,\\
&&[x_i,x_j]=1 \textrm{ for }|i-j|\neq t,\\\nonumber
&&[x_i,x_{t+i}]=z \textrm{ for  }0\leq i\leq t-1\rangle.
\end{eqnarray}
We note that $E_t$ is the extraspecial group of order $2^{2t+1}$ of ``plus type'', that is, the central product of $t$ dihedral groups $\D_4$.

We  define two group extensions (namely $G_t^+$ and $G_t^-$) of $E_t$  by the dihedral group
\begin{eqnarray}\label{eq:D}
\D_{2t}=\langle a,b\mid a^{2t}=b^2=1, a^b=a^{-1}\rangle.
\end{eqnarray}
In both extensions, the generators $a$ and $b$ of $\D_{2t}$ act upon the generators of $E_t$ according to the rules:
\begin{eqnarray*}
 x_i^a&=&x_{i+1}\>\quad \hbox{  for } 0\leq i\leq 2t-1,\>\\
x_i^b&=&x_{t-1-i}\> \hbox{  for } 0\leq i\leq 2t-1,
\end{eqnarray*}
where the indices are taken modulo $2t$.
To obtain the first extension $G_t^+$, we let $a^{2t} = b^2 =1$ (resulting in a semidirect product), while for the second extension 
 $G_t^-$ we let $a^{2t}=z$, $b^2 =1$ (resulting in a non-split extension):
$$
G_t^+ = E_t \rtimes \D_{2t}, \qquad
G_t^- = E_t.\D_{2t}.
$$
Finally, let $H_t^\pm$ be the subgroup of $G_t^\pm$ generated by the elements $\{x_0,\ldots,x_{t-1},b\}$
and observe that $H_t^+ \cong H_t^- \cong \C_2^t \rtimes \C_2$. Let
$$
 \Gamma_t^+ = \Cos(G_t^{+},H_t^+,a)\>  \hbox{ and }\>\Gamma_t^- = \Cos(G_t^{-},H_t^-,a).
$$
In Proposition~\ref{prop:someproperties}, we sum up some properties of $\Gamma_t^{\pm }$ which are proved in~\cite{gamma}.
\begin{proposition}\label{prop:someproperties}
The pairs $(\Gamma_t^+,G_t^+)$ and $(\Gamma_t^-,G_t^-)$ are locally-$\D_4$. 
For $t\ge 3$, the graphs $\Gamma_t^+$ and $\Gamma_t^-$ are
not isomorphic to the graphs $\C(r,s)$ for any $r$ and $s$ and satisfy $\Aut(\Gamma_t^\pm)=G_t^\pm$.
Finally, $\Gamma_2^+\cong\C(4,3)$, $G_2^+=\Aut(\Gamma_2^+)$, and $|\Aut(\Gamma_2^-):G_2^-|=9$.
\end{proposition}

\subsection{The graphs in Table~\ref{tb:soluble} and Table~\ref{tb:nsoluble}.}\label{ss:spor}
Most of the graphs in this section are obtained from standard graph operations applied to small $3$-valent arc-transitive graphs. We use the Foster Census notation~\cite{Foster} to denote  $3$-valent arc-transitive graphs. For instance, $F_6$ will denote the complete bipartite graph on $6$ vertices, $\Pet$  the Petersen graph,  $\Hea$ the Heawood graph, $F_{18}$ the Pappus graph, $\Tut$ the Tutte-Coxeter graph and $F_{90}$ the unique $3$-valent arc-transitive graph with $90$ vertices. The extensive census of 4-valent edge-transitive graphs of small order in~\cite{SW} is quite useful in understanding the graphs in this section.

Let $\Gamma$ be a graph. The \emph{bipartite double} of $\Gamma$, denoted $\B(\Gamma)$, is the categorical product $\Gamma\times\K_2$, with vertex set $\V(\Gamma)\times\{0,1\}$ and edges $\{(u,i),(v,1-i)\}$ for each edge $\{u,v\}$ of $\Gamma$. The \emph{line graph} of $\Gamma$, denoted $\Line(\Gamma)$, has edges of $\Gamma$ as vertices, with two such edges adjacent in $\Line(\Gamma)$ if they are adjacent in $\Gamma$. The \emph{arc graph} of $\Gamma$, denoted $\AG(\Gamma)$, has arcs of $\Gamma$ as vertices, with two such arcs $(u,v)$, $(v,w)$ adjacent in $\AG(\Gamma)$ if $u\neq w$. The $3$-\emph{arc graph} of $\Gamma$ (see~\cite{threearc}), denoted $\AAA(\Gamma)$, has arcs of $\Gamma$ as vertices, with two such arcs $(v_1,v_2)$, $(w_1,w_2)$ adjacent in $\AAA(\Gamma)$ if $w_1$ is adjacent to $v_1$ in $\Gamma$, $v_1\neq w_2$ and $v_2\neq w_1$. The \emph{hill capping} (see \cite{HillCap}) of $\Gamma$, denoted $\HC(\Gamma)$, has four vertices $\{u_0,v_0\}$, $\{u_0,v_1\}$, $\{u_1,v
 _0\}$, $\{u_1,v_1\}$, for each edge $\{u,v\}$ of $\Gamma$, and each $\{u_i,v_j\}$ is adjacent to each $\{v_j,w_{1-i}\}$, where $u$ and $w$ are distinct neighbors of $v$. Finally, the \emph{squared-arc graph} of $\Gamma$ is denoted $\AAG(\Gamma)$: the vertices of $\AAG(\Gamma)$ are the ordered pairs $((v_1,v_2),(w_1,w_2))$ with $(v_1,v_2)$ and $(w_1,w_2)$ arcs of $\Gamma$, and the edges of $\AAG(\Gamma)$ are the $2$-sets of the form $\{((v_1,v_2),(w_1,w_2)),((w_1,w_2),(v_2,v_3))\}$ with $v_1,v_2,v_3$ a $2$-arc in $\Gamma$.

Except for $\C_5\Box\C_5$ and $\C^{\pm 1}(3,3,3)$, all the graphs in Table~\ref{tb:soluble} and Table~\ref{tb:nsoluble} are obtained by starting with one of the $3$-valent graphs $F_6$, $\Pet$, $\Hea$, $F_{18}$, $\Tut$ or $F_{90}$ and applying some of the graph operations described above. The graph $\C_5\Box\C_5$ is the the cartesian product of two 5-cycles, while $\C^{\pm 1}(3,3,3)$ is one of an infinite family of $4$-valent graphs described in~\cite[Definition~$2.2$]{GarPra2}. For convenience, we define only $\C^{\pm 1}(3,3,3)$. Take $H=\langle m_1,m_2,m_3,g\mid m_i^3=[m_i,m_j]=1,g^3=m_1m_2m_3,m_1^g=m_2,m_2^g=m_3,m_3^g=m_1\rangle$ and define $\C^{\pm 1}(3,3,3)=\Cay(H,\{g,g^{-1},gm_1,(gm_1)^{-1}\})$.

We note that, for all but three pairs $(\Gamma,G)$ appearing in Table~\ref{tb:soluble} and Table~\ref{tb:nsoluble}, we have $G=\Aut(\Gamma)$ and hence the pair is uniquely determined by $\Gamma$. The three exceptional graphs are $\Line(\Tut)$, $\B(\Line(\Tut))$ and $\B(\Line(\Pet))$. If $\Gamma$ is one of $\Line(\Tut)$ or $\B(\Line(\Tut))$, then $|\Aut(\Gamma):G|\leq 2$. Finally, the full automorphism group of $\B(\Line(\Pet))$ is locally 2-transitive (in particular, this graph appears in~\cite[Table 3]{twoarc} as the graph $A[30,1]$), but it contains a subgroup of index 3 which is locally-$\D_4$.

\begin{table}
\begin{center}
\begin{tabular}{|c|c|c|c|c|}\hline
$\Gamma$ &$|\V\Gamma|$&$|G_v|$& $G$ \\\hline
$\Line(F_6)$& $9$&$8$&$\C_3^2\rtimes \D_4$\\
$\B(\Line(F_6))$& $18$&$8$&$(\C_3^2\rtimes \C_2)\rtimes \D_4$\\
$\C_5\Box\C_5$& $25$&$8$&$\C_5^2\rtimes \D_4$\\
$\Line(F_{18})$&$27$&$8$&$3^3_+\rtimes \D_4$\\
$\C^{\pm 1}(3,3,3)$&$81$&$16$&$(\C_3^3\rtimes \C_2)\rtimes \Sym(4)$\\\hline
\end{tabular}
\caption{Pairs in part~$(B)$ of Theorem~\ref{thm:main} with $G$ soluble}\label{tb:soluble}
\end{center}
\end{table}

\begin{table}
\begin{center}
\begin{tabular}{|c|c|c|c|c|c|c|}\hline
    &$\Gamma$ &$|\V\Gamma|$&$|G_v|$&$G$\\\hline
$(i)$&$\Line(\Pet)$& $15$&$8$&$\Sym(5)$\\
$(i)a$&$\AG(\Pet)$&$30$&$8$&$\Sym(5)\times \Sym(2)$\\
$(i)b$&$\Line(\B(\Pet))$&$30$&$8$&$\Sym(5)\times \Sym(2)$\\
$(i)c$&$\B(\Line(\Pet))$&$30$&$8$&$\Sym(5)\times \Sym(2)$\\
$(ii)$&$\Line(\Hea)$&$21$&$16$&$\PGL(2,7)$\\
$(ii)a$&$\B(\Line(\Hea))$&$42$&$16$&$\PGL(2,7)\times \Sym(2)$\\
$(ii)b$&$\HC(\Hea)$&$84$&$16$&$\PSL(2,7)\rtimes \D_4$\\
$(iii)$&$\Line(\Tut)$&$45$&$16$ or $32$& $G\trianglelefteq_{1,2}\PGammaL(2,9)$, $G\neq\Sym(6)$\\
$(iii)a$&$\B(\Line(\Tut))$&$90$&$16$ or $32$&$G\trianglelefteq_{1,2}\PGammaL(2,9)\times \Sym(2)$\\
$(iii)b$&$\AAA(\Tut)$&$90$&$16$&$\PGammaL(2,9)$\\
$(iii)c$&$\Line(F_{90})$&$135$&$32$&$3.\Alt(6).(\C_2^2)$\\
$(iii)d$&$\HC(\Tut)$&$180$&$32$&$\Sym(6)\rtimes \D_4$\\
$(iv)$&$\AG^2(\Tut)$&$8100$&$512$&$\PGammaL(2,9)\wr \Sym(2)$\\\hline
\end{tabular}
\caption{Pairs in part~$(B)$ of Theorem~\ref{thm:main} with $G$ not soluble\newline
\footnotesize (The notation $H\trianglelefteq_{1,2}K$  means that $H\trianglelefteq K$ and that $|K:H|=1$ or $2$.)}\label{tb:nsoluble}
\end{center}
\end{table}

We will need the following two results about the pairs appearing in Table~\ref{tb:soluble} and Table~\ref{tb:nsoluble}.
\begin{lemma}\label{thm:exp2}
Let $(\Gamma,G)$ be a locally-$\D_4$ pair. Assume that $G$ has an abelian minimal normal subgroup $N$ such that $\Gamma/N$ is $4$-valent and that $(\Gamma/N,G/N)$ is one of the pairs in Table~\ref{tb:soluble}. Then either $|\V\Gamma|>2|G_v|\log_2(|G_v|/2)$, $\Gamma=\C(9,1)$ or $(\Gamma,G)$ is one of the pairs in Table~\ref{tb:soluble}.
\end{lemma}

\begin{proof}
As $\Gamma/N$ is $4$-valent, the group $N$ acts semiregularly on $\V\Gamma$, $G/N$ acts faithfully on $\V(\Gamma/N)$, $(\Gamma/N,G/N)$ is locally-$\D_4$, the vertex-stabiliser in $G/N$ is isomorphic to $G_v$ and $|\V\Gamma|=|\V(\Gamma/N)||N|$. We may assume that $|\V(\Gamma/N)||N|\leq 2|G_v|\log_2(|G_v|/2)$. Since $|N|\geq 2$, a direct inspection of the pairs in Table~\ref{tb:soluble} reveals that we must have $\Gamma/N=\Line(F_6)$ and hence $9|N|\leq 32$. In particular, we may assume that $|N|\leq 3$. The rest of the proof is computational with the help of \texttt{Magma}~\cite{magma}. If $|N|=2$, then either $\Gamma=\C(9,1)$ or $\Gamma=\B(\Line(F_6))$. If $|N|=3$, then $\Gamma=\Line(F_{18})$.
\end{proof}

\begin{lemma}\label{thm:exp}
Let $(\Gamma,G)$ be a locally-$\D_4$ pair. Assume that $G$ has a minimal normal subgroup $N$ such that $\Gamma/N$ is $4$-valent and that $(\Gamma/N,G/N)$ is one of the pairs in Table~\ref{tb:nsoluble}. Then either $|\V\Gamma|>2|G_v|\log_2(|G_v|/2)$ or $(\Gamma,G)$ is one of the pairs in Table~\ref{tb:nsoluble}.
\end{lemma}
\begin{proof}
As $\Gamma/N$ is $4$-valent, the group $N$ acts semiregularly on $\V\Gamma$, $G/N$ acts faithfully on $\V(\Gamma/N)$, $(\Gamma/N,G/N)$ is locally-$\D_4$, the vertex-stabiliser in $G/N$ is isomorphic to $G_v$ and $|\V\Gamma|=|\V(\Gamma/N)||N|$. We may assume that $|\V(\Gamma/N)||N|\leq 2|G_v|\log_2(|G_v|/2)$. Since $|N|\geq 2$, a direct inspection of the pairs in Table~\ref{tb:nsoluble} reveals that we must have that the pair $(\Gamma/N,G/N)$ is either in row $(i)$ with $|N|=2$, in row $(ii)$ with $|N|\leq 4$, in row $(ii)a$ with $|N|=2$, in row $(iii)$ with $|N|\leq 5$ and $|G_v|=32$, or in row $(iii)a$ with $|N|=2$ and $|G_v|=32$. 

We use \texttt{Magma} to deal with these cases. If $(\Gamma/N,G/N)$ is in row $(i)$, $(ii)$ or $(ii)a$, then $|G|=|G/N||N|<2000$ and hence $G$ must appear in the \texttt{SmallGroups} database
of \texttt{Magma}.  For each candidate group $G$, we compute the list of core-free
subgroups $Q$ of $G$ of order $|G_vN/N|$ and we construct the permutation
representation of $G$ on the right cosets of $Q$ in $G$. Finally, we check
whether there exists a self-paired suborbit of size $4$ giving rise to a
connected locally-$\D_4$ pair. The only pairs arising in this way are already in Table~\ref{tb:nsoluble}.

We now assume that $(\Gamma/N,G/N)$ is in row $(iii)$ (respectively $(iii)a$) and hence $G/N$ is isomorphic to $\PGammaL(2,9)$ (respectively $\PGammaL(2,9)\times \Sym(2)$). Consider the socle $S/N=\soc(G/N)$. We have $S/N\cong \Alt(6)$ (respectively $S/N\cong \Alt(6)\times \C_2$) and $S/N$ is transitive
on $V(\Gamma/N)$. Therefore $S$ acts transitively on $V\Gamma$ and $|S|=|S/N||N|\leq
2000$. In particular, the group $S$ can be found in
the \texttt{SmallGroups} database. It can be checked that the stabiliser of the
vertex $v^N$ in $S/N$ has two orbits on $\Gamma(v^N)$ and hence $S_v$ has
two orbits on $\Gamma(v)$. For each candidate $S$, we compute the list of core-free
subgroups $Q$ of $S$ of order $|S_vN/N|$ and we construct the permutation
representation of $S$ on the right cosets of $Q$ in $S$. Finally we check
whether there exists two distinct self-paired suborbits of size $2$ whose
union gives rise to a connected locally-$\D_4$ pair. The only pairs arising in this way are already in Table~\ref{tb:nsoluble}.
\end{proof}

\section{$G$ has an abelian minimal normal subgroup $N$}\label{sec:abelian}
\subsection{$\Gamma/N$ has valency at most 2}\label{sec:basic}
The case when the quotient is a cycle was examined in some details in~\cite{GarPra2} and we report some results that follow from their work.

\begin{theorem}[\cite{GarPra2}, Theorem 1.1, Lemma 3.1]
\label{GarPraCycle}Let $\Gamma$ be a connected $4$-valent $G$-arc-transitive graph and let $N$ be a minimal normal $p$-subgroup of $G$ with orbits of size $p^s$, for some prime $p$. Let $K$ denote the kernel of the action of $G$ on the $N$-orbits. Suppose that the quotient $\Gamma/N$ is a cycle of length $r\geq 3$. Then either $G$ has an abelian normal subgroup that is not semiregular on the vertices of $\Gamma$, or $p$ is odd and $K_v$ is a nontrivial elementary abelian $2$-group of order dividing $2^s$.
\end{theorem}

One of the cases in the conclusion of Theorem~\ref{GarPraCycle} is that $G$ has an abelian normal subgroup that is not semiregular on the vertices of $\Gamma$. It turns out that this is a very strong restriction, as seen in the following theorem, which is \cite[Theorem~$1$ with $p=2$ ]{PraegerXu}.

\begin{theorem}\label{PXu}
Suppose that $\Gamma$ is a connected $4$-valent $G$-arc-transitive graph, and that $G$ has an abelian normal subgroup which is not semiregular on the vertices of $\Gamma$. Then $\Gamma\cong \C(r,s)$ for some $r\geq \max\{3,s+1\}$, $s\geq 1$. 
\end{theorem}

As noted in the introduction, there is an obvious action of $\C_2^r\rtimes \D_r$ on $\C(r,s)$. It turns out that, 
when $r\neq 4$, this is in fact the full automorphism group.

\begin{theorem}[\cite{PraegerXu}, Theorem~$2.13$]\label{PXuAutomorphism}
Let $\Gamma=\C(r,s)$ and let $H=\C_2^r\rtimes\D_r$. If $r\neq 4$, then $\Aut(\Gamma)=H$. Moreover, $|\Aut(\C(4,1)):H|=9$, $|\Aut(\C(4,2)):H|=3$ and $|\Aut(\C(4,3)):H|=2$.
\end{theorem}

Combining the two previous theorems, we get the following locally-$\D_4$ version of Theorem~\ref{PXu}, which will be used repeatedly.

\begin{corollary}\label{RealCorollary}
Let $(\Gamma,G)$ be a locally-$\D_4$ pair and let $v$ be a vertex of $\Gamma$. If $G$ has an abelian normal subgroup which is not semiregular on the vertices of $\Gamma$, then
$\Gamma\cong \C(r,s)$ for some $1\leq s\leq r-2$. Moreover, if $|\V\Gamma|\leq 2|G_v|\log_2(|G_v|/2)$, then $s\leq \frac{r}{2}$.
\end{corollary}
\begin{proof}
By Theorem~\ref{PXu}, we have $\Gamma\cong \C(r,s)$ for some $r\geq \max\{3,s+1\}$, $s\geq 1$. To show the first claim, it suffices to show that $r\neq s+1$. Suppose, on the contrary, that $r=s+1$. If $r\neq 4$, then it follows from Theorem~\ref{PXuAutomorphism} that $|\Aut(\C(r,r-1))_v|=|\C_2^r\rtimes \D_r|/|\V(\C(r,r-1))|=2^{r+1}r/(2^{r-1}r)=4$ and hence $(\Gamma,G)$ cannot be locally-$\D_4$, which is a contradiction. If $r=4$, then $s=3$. By Theorem~\ref{PXuAutomorphism}, $|\Aut(\C(4,3))_v|=8$ and hence $G=\Aut(\C(4,3))$. It can be checked that every abelian normal subgroup of $\Aut(\C(4,3))$ acts semiregularly on the vertices of $\C(4,3)$, which contradicts the hypothesis of Corollary~\ref{RealCorollary}.

We now assume that $|\V\Gamma|\leq 2|G_v|\log_2(|G_v|/2)$ and show that $s\leq \frac{r}{2}$. From the first part of the proof, we have $1\leq
s\leq r-2$ and hence we may assume that $r>4$. Suppose, on the contrary, that $s>\frac{r}{2}$ and hence $r-s\leq s-1$. It follows from Theorem~\ref{PXuAutomorphism} that $|G_v|\leq |\Aut(\Gamma)_v|=|\C_2^r\rtimes \D_r|/|\V\Gamma|=r2^{r+1}/(r2^s)=2^{r-s+1}$ and hence $2|G_v|\log_2(|G_v|/2)\leq (r-s)2^{r-s+2}$. If $r-s=s-1$, then $2|G_v|\log_2(|G_v|/2)\leq (s-1)2^{s+1}<(2s-1)2^s=r2^s=|\V\Gamma|$, which is a contradiction. Hence, we may assume that $r-s\leq s-2$ and hence $2|G_v|\log_2(|G_v|/2)\leq (r-s)2^s<r2^s=|\V\Gamma|$, which is also a contradiction.
\end{proof}

The main result of this section is the following.

\begin{theorem}\label{thm:mainbasic}
Let $(\Gamma,G)$ be a locally-$\D_4$ pair. Assume that $G$ has an abelian minimal normal subgroup $N$ such that $\Gamma/N$ has valency at most $2$. Then one of the following holds:
\begin{description}
\item[$(A)$]$\Gamma\cong \C(r,s)$ for some $r\geq 3$, $1\leq s\leq \frac{r}{2}$;
\item[$(B)$]$(\Gamma,G)$ is one of the pairs in Table~\ref{tb:soluble} with $\Gamma\neq \Line(F_{18})$;
\item[$(C)$]$|\V\Gamma|>2|G_v|\log_2(|G_v|/2)$.
\end{description}
\end{theorem}

\begin{proof}
If $G$ has an abelian normal subgroup that is not semiregular on $\V\Gamma$, then, by Corollary~\ref{RealCorollary}, part~$(A)$ or~$(C)$ holds. We will therefore assume that every abelian normal subgroup of $G$ acts semiregularly on $\V\Gamma$. Write $|N|=p^s$, for some prime $p$ and $s\geq 1$. 

Suppose first that $\Gamma/N$ is a cycle of length $r\geq 3$. Let $K$ denote the kernel of the action of $G$ on the $N$-orbits.   Since $\Gamma/N$ is a cycle and $(\Gamma,G)$ is locally-$\D_4$, we have $|G_v|=2|K_v|$. By Theorem~\ref{GarPraCycle}, we obtain $|\V\Gamma|= rp^s$, $p\geq 3$ and $|K_v|$ divides $2^s$. In particular, $|G_v|\leq 2^{s+1}$ and $s\geq 2$. If $|\V\Gamma|>2|G_v|\log_2(|G_v|/2)$, then part~$(C)$ holds, hence we may assume that $rp^s\leq 2|G_v|\log_2(|G_v|/2)\leq s2^{s+2}$. A simple examination of the cases reveals that we must have $p=r=3$, $s\in\{2,3,4\}$ and $|G_v|=2^{s+1}$. In particular, $|K_v|=2^s$. The graphs for which the equality $|K_v|=2^s$ is satisfied are classified in~\cite[Theorem~$1.1$~$(b)$]{GarPra2} and a direct inspection of these graphs gives  $s=3$, $\Gamma=\C^{\pm 1}(3,3,3)$ and part~$(B)$ follows.

Assume now that $\Gamma/N \cong \K_1$ or $\K_2$. This case was considered already by Gardiner and Praeger~\cite{GarPra1}. To avoid a tedious consideration of all the cases appearing in their classification, we give an independent argument. If $p=2$, then (as $G_v$ is a $2$-group) $G$ is a $2$-group. By minimality of $N$, we get $|N|=2$,  and hence $|\V\Gamma|\le 4$, which is a contradiction. Assume now that $p$ is odd.
We show that in this case the kernel $K(v)$ of the action of $G_v$ on $\Gamma(v)$ is trivial, that is, $G_v$ acts faithfully on $\Gamma(v)$.
Since $N$ is semiregular, $N$ has precisely $4$ orbits on the arcs  of $\Gamma$ if  $\Gamma/N \cong \K_1$ (respectively, $N$ has precisely $4$ orbits on the edges  of $\Gamma$ if $\Gamma/N \cong \K_2$).
Note that $K(v)$ fixes each of these $N$-orbits setwise. On the other hand, for each vertex $u$, every element of a vertex-stabiliser $G_u$ which fixes each of the $N$-orbits on arcs (respectively, edges) setwise is in $K(u)$.
By connectivity of $\Gamma$, this implies that $K(v)$ fixes every vertex of $\Gamma$, and hence is trivial. In particular, this shows that $G_v\cong \D_4$. If $|\V\Gamma| > 32$, then part~$(C)$ holds.
We shall therefore assume that $|\V\Gamma| \le 32$.

Consider the action of $G_v$ on $N$ by conjugation. If this action is not faithful, then a nontrivial element $x$ of $G_v$ centralising $N$ fixes every vertex in the $N$-orbit $v^N$.
In particular, $N$ has two orbits on $\V\Gamma$ forming a bipartition of $\Gamma$. Since $x\not =1$, this implies that $u^x \neq u$ for a neighbour $u$ of $v$, and in particular, $u$ and $u^x$
share the same neighbourhood. It is then easy to show that $\Gamma \cong \C(r,1)$ for some $r\ge 3$ (see~\cite[Lemma 4.3]{PW}), and part~$(A)$ follows.

We may therefore assume that $G_v$ acts faithfully on $N$ by conjugation, that is, $\Aut(N)$ contains a subgroup isomorphic to $G_v\cong \D_4$. In particular, since the automorphism group of a group of prime order is cyclic, we have $s\ge 2$.
If $\Gamma/N \cong \K_2$, then $32 \ge |\V\Gamma| = 2 p^s$, and hence $p=3$ and $s=2$. It easy to see that $\Gamma \cong \B(\Line(F_6))$ from which part~$(B)$ follows.
On the other hand, if $\Gamma/N \cong \K_1$, then $N$ acts regularly on $\V\Gamma$ and therefore $\Gamma=\Cay(N,S)$ for some inverse-closed generating subset $S$ of $N$. 
In particular, since $|S|=4$, $N$ is generated by $2$ elements, and hence $s=2$. Moreover, since $32 \ge |\V\Gamma| = p^2$, it follows that $p\in \{3,5\}$.
It is then easy to see that $\Gamma=\Line(F_6)$ or $\C_5\Box\C_5$ from which part~$(B)$ follows.
\end{proof}

\subsection{$\Gamma/N\cong\C(r,s)$}\label{sec:nc}
In this section we deal with the case where the quotient $\Gamma/N$ by the abelian minimal normal subgroup $N$ of $G$ is isomorphic to $\C(r,s)$ for some $r\geq 3$, $1\leq s\leq \frac{r}{2}$. We first need the following lemma, which is a kind of converse to Corollary~\ref{RealCorollary}.

\begin{lemma}\label{lemma:againC}
Let $(\Gamma,G)$ be a locally-$\D_4$ pair. If $\Gamma\cong\C(r,s)$ for some $1\leq s\leq r-2$, then $G$ has an elementary abelian normal $2$-subgroup $A$ such that $A$ is not semiregular on $\V\Gamma$, $\Gamma/A$ is a cycle of length $m$ for some multiple $m$ of $r$, $G/A\cong \D_m$ and $A$ is equal to the normal closure $A_v^G$ of $A_v$ in $G$ (where $v\in \V\Gamma$). 
\end{lemma}
\begin{proof}
As noted in the introduction, there is an obvious action of $H=B\rtimes \D_r$ with $B=\C_2^r$ on $\Gamma$. Note that $\Gamma/B$ is a cycle of length $r$. Suppose that $r\neq 4$. Then, it follows from Theorem~\ref{PXuAutomorphism} that $H=\Aut(\Gamma)$ and hence $G\leq H$. Moreover $|H_v:B_v|=2$ and hence $|G\cap H_v:G\cap B_v|=|G_v:G\cap B_v|\leq 2$. Let $A$ be the normal closure of $G\cap B_v$ in $G$. Note that, as $A\leq B$, the group $A$ is an elementary abelian normal 2-subgroup of $G$. Since $(\Gamma,G)$ is locally-$\D_4$, it follows that $G_v$ is not contained in $B_v$ and hence $|G_v:G\cap B_v|=2$. Since $G\cap B_v$ is contained in $A_v$, it follows that $|G_v:A_v|=2$, $G\cap B_v=A_v$ and $A=A_v^G$. Since $|G_v|\geq 8$, we have $|A_v|\geq 4$ and hence $A$ is not semiregular on $\V\Gamma$. In particular, $\Gamma/A$ has valency at most 2. Since $A\leq B$ and $\Gamma/B$ is a cycle of length $r\geq 3$, it follows that $\Gamma/A$ is a cycle of length a multiple $m$ of $r$. 
 Finally, since $G/A$ acts arc-transitively on $\Gamma/A$
  and $|G_v:A_v|=2$, it follows that $G/A$ acts faithfully on $\Gamma/A$ and hence $G/A\cong\D_m$.

Assume now that $r=4$ and, in particular, $s\in\{1,2\}$. By Theorem~\ref{PXuAutomorphism}, $H$ is a Sylow $2$-subgroup of $\Aut(\Gamma)$. Since $(\Gamma,G)$ is locally-$\D_4$ and $|\V\Gamma|$ is a power of $2$, the group $G$ is a $2$-group and hence $G$ is conjugate to a subgroup of $H$. The rest of the proof is as in the previous paragraph.
\end{proof}

The next two lemmas are simply technical and well-known, but we include a proof for the sake of completeness.

\begin{lemma}\label{Pablolemma2}
Let $P=\langle x_0,\ldots,x_n\rangle$ be a $p$-group. If, 
for some $0\leq i\leq n-1$ and $g\in P$, we have $x_n=x_i^g$, then $P=\langle
x_0,\ldots,x_{n-1}\rangle$.   
\end{lemma}

\begin{proof}
We recall that $g$ is called a non-generator of $P$ if, for any subset $X$ of $P$,
$P=\langle g,X\rangle$ implies that $P=\langle X\rangle$. In a
$p$-group, every commutator is a non-generator, see~\cite[$5.3.2$]{DR}. Assume $x_n=x_i^g$, for some $0\leq i\leq n-1$ and $g\in P$. We have
\begin{eqnarray*}
P&=&\langle x_0,\ldots,x_{n-1},x_n\rangle=\langle
x_0,\ldots,x_{n-1},x_i^g\rangle=\langle x_0,\ldots,x_{n-1},x_i[x_i,g]\rangle\\ 
&=&\langle x_0,\ldots,x_{n-1},[x_i,g]\rangle=\langle x_0,\ldots,x_{n-1}\rangle.
\end{eqnarray*}
\end{proof}

\begin{lemma}\label{PabloLemma}
Let $q$  be an odd prime power and let $H$ be an elementary abelian 
$2$-subgroup of $\GL(n,q)$ of order $2^r$. Then $H$ is conjugate to a subgroup of the  group consisting of the scalar matrices. In particular, $r\leq n$.
\end{lemma}
\begin{proof}
Let $D$ be the subgroup of scalar matrices of $\GL(n,q)$. Consider a vector space $V$ of dimension $n$ over $\mathbb{F}_q$ and write $H=\langle h_1,\ldots,h_r\rangle$. We will show that there exists a decomposition $V=\oplus_{i=1}^kV_k$ such that the action of $H$ on $V_i$ is given by the multiplication by $\pm 1$. Note that this implies that $H$ is conjugate to a subgroup of $D$. The proof is by induction on $r$. When $r=0$, there is nothing to prove. Suppose that $r\geq 1$ and let $V_+=\{v+vh_1\mid v\in V\}$, $V_{-}=\{v-vh_1\mid v \in V\}$ and let $v\in V$. Since $q$ is odd, we have $$v=\left(\frac{v}{2}+\frac{v}{2}h_1\right)+\left(\frac{v}{2}-\frac{v}{2}h_1\right)\in V_{+}+V_{-}$$ and hence $V=V_{+}+V_{-}$. Using the fact that $h_1^2=1$, it is easy to check that $V_{+}$ is the eigenspace corresponding to the eigenvalue $1$ of $h_1$ and that $V_{-}$ is the eigenspace corresponding to the eigenvalue $-1$ of $h_1$. Let $v\in V_+$ and $h\in H$. As $H$ is abelian, we have $vh=v
 h_1h=(vh)
 h_1$ and hence $vh\in V_{+}$. This yields that $V_+$ is an $H$-submodule of $V$. Similarly, $V_{-}$ is an $H$-submodule of $V$. The claim follows by considering the action of $\langle h_2,\ldots,h_r\rangle$ on $V_+$ and on $V_{-}$ and using the induction hypothesis. Finally, since $D$ is the direct product of $n$ cyclic groups of order $q-1$ and $H$ is conjugate to a subgroup of $D$, we must have $r\leq n$.
\end{proof}

The next theorem is the main result of this section and a key ingredient in the proof of Theorem~\ref{thm:main}.

\begin{theorem}\label{WreathCover}
Let $(\Gamma,G)$ be a locally-$\D_4$ pair. Assume that $G$ has an abelian minimal normal $p$-subgroup $N$ such that $\Gamma/N\cong \C(r,s)$ for some $r\geq 3$, $1\leq s\leq \frac{r}{2}$. Then one of the following holds:
\begin{description}
\item[$(A)$]$\Gamma\cong \C(r',s')$ for some $r'\geq 3$, $1\leq s'\leq \frac{r'}{2}$;
\item[$(B)$]$p=2$ and $|\V\Gamma|\geq2|G_v|\log_2(|G_v|/2)$;
\item[$(C)$]$p \ge 3$ and $|\V\Gamma|\geq \frac{6}{p}|G_v|^{\log_2(p)}$.
\end{description}
Moreover, if the inequality in~$(B)$ holds with equality and $\Gamma$ is not a graph as in~$(A)$, then $(\Gamma,G)$ is one of $(\Gamma_t^+,G_t^+)$ or $(\Gamma_t^-,G_t^-)$ for some $t\geq 2$.
\end{theorem}
\textbf{Remark:} Note that for $x\geq 8$ and $p\geq 3$, we have $\frac{6}{p}x^{\log_2(p)}>2x\log_2(x/2)$ and hence if~$(C)$ holds, then $|\V\Gamma|>2|G_v|\log_2(|G_v|/2)$ (the same inequality as in~$(B)$).

\begin{proof}
As $\Gamma/N$ is $4$-valent, we obtain that $N$ is semiregular on $\V\Gamma$. Furthermore, as $\Gamma/N\cong \C(r,s)$ for some $r\geq 3$, $1\leq s\leq \frac{r}{2}$, it follows that $s\leq r-2$ and hence Lemma~\ref{lemma:againC} implies that $G/N$ contains an elementary abelian normal $2$-subgroup $E/N$ not semiregular on $\V(\Gamma/N)$ with $(\Gamma/N)/(E/N)\cong \Gamma/E$ a cycle of length $m\geq 3$, $(G/N)/(E/N)
\cong G/E\cong \D_m$ and $E/N$ is equal to the normal closure $(E_vN/N)^{G/N}$ of $E_vN/N$ in $G/N$. Let $F$ be the normal closure of $E_v$ in $G$. As $N$ is a minimal normal subgroup of $G$, we obtain that either $N\cap F=1$ or $N\leq F$. If $N\cap F=1$, then $F\cong FN/N\leq E/N$ is a normal elementary abelian $2$-subgroup of $G$ not acting semiregularly on $\V\Gamma$, and hence by Corollary~\ref{RealCorollary} we obtain that~$(A)$ or~$(B)$ holds. Therefore we may assume that $N\leq F$. As the normal closure of $E_vN$ in $G$ is $E$, we have $E=(E_vN)^G=E_v^GN=FN=F$, that is, $E=E_v^G$. 

Since $\Gamma/E$ is a cycle of length $m\geq 3$ and $G/E\cong \D_m$, it follows that $|G_v|=2|E_v|$ and $|\V\Gamma|=m|v^E|$ for some vertex $v\in\V\Gamma$. As $E/N$ is an elementary abelian $2$-group and $N$ is semiregular, $ E_vN/N\cong E_v$ is an elementary abelian $2$-group.

Suppose first that $p\geq 3$. Write $|E_v|=2^t$. Let $C=\C_E(N)$ be the centraliser of $N$ in $E$. Since $N$ and $E$ are normal in $G$, so is $C$. As $(|N|,|C:N|)=1$,  by the Schur-Zassenhaus theorem, $N$ has a complement $K$ in $C$, that is, $C=NK$ and $N\cap K=1$. Since $K$ centralises $N$, we have $C=N\times K$. It follows that $K$ is characteristic in $C$ and hence normal in $G$. Furthermore as $E/N$ is abelian, the group $K$ is abelian and so is $C$. If $C$ does not act semiregularly on $\V\Gamma$, then $K$ does not act semiregularly on $\V\Gamma$ and, from Corollary~\ref{RealCorollary}, we obtain that~$(A)$ or~$(B)$ holds.  Hence we may assume that $C$ acts semiregularly on $\V\Gamma$. In particular, $1=E_v\cap C=\C_{E_v}(N)$ and the elementary abelian $2$-group $E_v$ acts faithfully on $N$ by conjugation. As $|E_v|=2^t$, Lemma~\ref{PabloLemma} implies  $|N|\geq p^t$. Hence $$|\V\Gamma|=|\V(\Gamma/N)||N|\geq 6|N|\geq 
 6p^
 t=\frac{6}{p}(2^{t+1})^{\log_2(p)}=\frac{6}{p}(2|E_v|)^{\log_2(p)}=\frac{6}{p}|G_v|^{\log_2(p)}$$ and part~$(C)$ holds.

From now on, we assume that $p=2$. In particular, $E$ is a $2$-group. Fix an orientation of the cycle $\Gamma/E \cong \C_m$, thus obtaining a directed cycle $\vC_m$. By lifting this orientation to the graph $\Gamma$,
we obtain a digraph $\vGa$ of in-degree and out-degree $2$, whose underlying graph is $\Gamma$, and such that $\vGa/E \cong \vC_m$. Observe that the orientation preserving group $G^+ = \Aut(\vGa) \cap G$ 
has index $2$ in $G$, contains the group $E$ and the quotient group $G^+/E$ is cyclic of order $m$.

Let $v$ be a vertex of $\vGa$. Let $t$ be the largest integer such that $E_v$ acts transitively on the $t$-arcs of $\Gamma$ starting at $v$ and let $(v_0,...,v_t)$, $v_0 = v$, be such a $t$-arc. For $0\leq i\leq t$, let $E_i$ be the pointwise stabiliser of $\{v_0,...,v_{t-i}\}$. Consider the action of $E_0$ on the out-neighbours of $v_t$. If this action were transitive, then $E_v$ would act transitively on the $(t+1)$-arcs starting at $v$, contradicting the maximality of $t$. Since $v_t$ has only two out-neighbours, we conclude that $E_0$ must fix them both. Since $\Gamma$ is strongly connected, it follows that $E_0=1$ and hence $|E_i|=2^i$ for $0\le i \le t$. In particular,
$|E_t|=|E_{v}|=2^t$ and 
\begin{equation}\label{eq:Thm241}
|G_{v}|=2|E_v|=2^{t+1}.
\end{equation}

As $|G_v|\geq 8$, we have $t\geq 2$. Since $E_v$ is transitive on the $t$-arcs of $\vGa$ starting at $v$ and $G^+$ is vertex-transitive, $G^+$ is transitive on $t$-arcs of $\vGa$. In particular, there exists $a\in G^+$ such that $(v_0,\ldots,v_t)^a=(v_0^a,v_0,\ldots,v_{t-1})$, that is, $v_i = v_0^{a^{-i}}$ for $0\leq i\leq t$. As $a$ acts as a rotation of order $m$ on $\vGa/E$, we get $G^+=E\langle a\rangle$. Let $x$ be the generator of the cyclic group $E_1$. For any integer $i$, let $x_i=x^{a^i}$ and $v_i=v_0^{a^{-i}}$ (note that this definition of $v_i$ is consistent with the definition of $v_i$ that we had for $0\leq i\leq t$). 

To make the rest of the proof easier to read, we prove six claims from which the theorem will follow.

\smallskip

\noindent\textsc{Claim~1. }$E_i=\langle x_0,...,x_{i-1}\rangle$ for $1\leq i\leq t$. 

\noindent We argue by induction on $i$. If $i=1$, then by definition, $x=x_0$ and $E_1=\langle x_0\rangle$. Assume $E_i=\langle x_0,\ldots,x_{i-1}\rangle$ for some $i$ with $1\leq i\leq t-1$. As $x$ fixes $\{v_0,\ldots,v_{t-1}\}$ pointwise  and $v_t^{x}\neq v_t$, the element $x_i=x^{a^i}$ fixes $\{v_0^{a^i},\ldots,v_{t-1}^{a^i}\}$ pointwise  and $(v_t^{a^i})^{x_i}\neq v_t^{a^i}$, that is, $x_i$ fixes $\{v_{-i},\ldots,v_{-i+t-1}\}$ pointwise and $v_{-i+t}^{x_i}\neq v_{-i+t}$. In particular, by definition of $E_{i+1}$, we get  $x_i\in E_{i+1}\setminus E_i$. As $|E_{i+1}:E_i|=2$, we obtain $E_{i+1}=E_i\langle x_i\rangle=\langle x_0,\ldots,x_i\rangle$, completing the induction.~$_\blacksquare$

\smallskip

For any positive integer $i\geq 1$, we define $E_i=\langle x_0,\ldots,x_{i-1}\rangle$ (Claim~$1$ shows that, for $1\leq i\leq t$, this definition is consistent with the original definition of $E_i$). Note that, for any $i\geq 0$, $E_i\leq \langle E_i,E_i^a\rangle=E_{i+1}$. Since $E$ is finite, there exists a smallest $e\geq 0$ such that $E_{t+e}=E_{t+e+1}$. Since $E_{t+e}=E_{t+e+1}=\langle E_{t+e},E_{t+e}^a\rangle$, it follows that $E_{t+e}$ is normalised by $a$.
\smallskip

\noindent\textsc{Claim~$2$. }$E=E_{t+e}$. 

\noindent Clearly $E_{t+e}\leq E$. Moreover, since $\vGa$ is a connected $G^+$-arc-transitive digraph and $a$ maps $v$ to an adjacent vertex, we have that $G^+=\langle G^+_v,a\rangle=\langle E_t,a\rangle$. It follows that $E_{t+e}$ is normalised by $G^+$. Therefore, as $E_v=E_t$, we obtain $E_{t+e}\geq E_v^{G^+}=\langle E_w\mid w\in \V\Gamma\rangle=E_v^{G}=E$.~$_\blacksquare$

\smallskip

From the definition of $e$, we have $|E_{t+i}:E_{t+i-1}|\geq 2$ for $1\leq i\leq e$ and hence $|E_{t+e}:E_t|\geq 2^e$. In particular, Claim~$2$ gives

\begin{equation}\label{eq:Thm242}
|v^{E}|=|E:E_v|=|E_{t+e}:E_t|\geq 2^e.
\end{equation} 

\noindent\textsc{Claim~$3$. }$m\geq t+e$.

\noindent Assume, by contradiction, that $m<t+e$. In particular,  $E=E_{t+e}=\langle
x_0,\ldots,x_{t+e-m-1},\ldots,x_{t+e-1}\rangle$. Since $G^+/E$ is a cyclic group of order $m$ and $a\in G^+$, we get
$a^m\in E$ but $x_{t+e-1}=x_{t+e-m-1}^{a^m}$ and hence, by
Lemma~\ref{Pablolemma2}, we have $E_{t+e}=\langle x_0,\ldots, 
x_{t+e-2}\rangle=E_{t+e-1}$, contradicting  the minimality of $e$.~$_\blacksquare$

\smallskip

Let  $\Z(E)$ be the centre of $E$. If $\Z(E)$ does not act semiregularly on $\V\Gamma$, then, by Corollary~\ref{RealCorollary}, we obtain that~$(A)$ or~$(B)$ holds. Therefore we may assume that $\Z(E)$ acts semiregularly on $\V\Gamma$. Recall that $E_v=E_t=\langle x_0,\ldots,x_{t-1}\rangle$ is abelian and hence $E_t^{a^{t-1}}=\langle x_{t-1},\ldots,x_{2t-2}\rangle$ is also abelian. Therefore $x_{t-1}$ is central in $\langle E_t,E_t^{a^{t-1}}\rangle=\langle x_0,\ldots,x_{2t-2}\rangle=E_{2t-1}$. Since $x_{t-1}\in E_v$ and $\Z(E)\cap E_v=1$, we get $E_{2t-1}<E=E_{t+e}$ and hence $2t-1<t+e$ from which it follows that $e\geq t$. Assume $e\geq t+1$. From~$(\ref{eq:Thm242})$ and Claim~$3$, we have $|\V\Gamma|=m|v^E|\geq (2t+1)2^{t+1}$. From~$(\ref{eq:Thm241})$, we have $2|G_v|\log_2(|G_v|/2)=2t2^{t+1}$ and hence~$(B)$ holds with the inequality being strict. Therefore, from now on, we may assume that $e=t$ and, in particular,
\begin{equation}\label{eq:6}
E=E_{2t}=\langle x_0,\ldots,x_{t-1},x_t,\ldots,x_{2t-1}\rangle=\langle E_v,E_v^{a^t}\rangle.
\end{equation}

Since $E$ is a $2$-group, $\Z(E)$ intersects every normal subgroup of $E$ non-trivially. In particular, $N\cap \Z(E) \not =1$.  Since $N$ is a minimal normal subgroup of $G$ and $\Z(E)$ is normal in $G$, this implies that
$N\leq \Z(E)$. Let $y$ be in $\Z(E)E_v^{a^t}\cap E_v$. Then $y=nx$ for some $n\in \Z(E)$ and $x\in E_v^{a^t}$. Let $g$ be in $E_v^{a^t}$. Using the fact that $n$ is central in $E$, that $x$ and $g$ are in $E_v^{a^t}$ and that $E_v^{a^t}$ is abelian, we see that
$y=nx$ commutes with $g$. Since $g$ is an arbitrary element of $E_v^{a^t}$, we obtain that $y$ is centralised by $E_v^{a^t}$. As $E_v$ is abelian and $y\in E_v$, the element $y$ is centralised by $E_v$. Hence, by~$(\ref{eq:6})$, we obtain $y\in \Z(E)\cap E_v=1$ and therefore $\Z(E)E_v^{a^t}\cap E_v=1$. It follows that 
\begin{eqnarray}\label{eq:fact0}
|\Z(E)E_v^{a^t}E_v|&=&\frac{|\Z(E)E_v^{a^t}||E_v|}{|\Z(E)E_v^{a^t}\cap E_v|}=|\Z(E)E_v^{a^t}||E_v|\\
&=&\frac{|\Z(E)||E_v^{a^t}|}{|\Z(E)\cap E_v^{a^t}|}2^t=|\Z(E)||E_v^{a^t}|2^t=|\Z(E)|2^{2t}.\nonumber
\end{eqnarray} 

In particular, $|E|\geq |\Z(E)|2^{2t}\geq 2^{2t+1}$ and hence $$|\V(\Gamma)|=m|v^E|=m|E:E_v|\geq 2t 2^{t+1}=2|G_v|\log_2(|G_v|/2)$$ and part~$(B)$ holds, with equality if and only if $m=2t$ and $|E|=2^{2t+1}$. This concludes the proof of the first part of Theorem~\ref{WreathCover}. 

For the remainder of this proof, we assume that~$(B)$ holds with equality and that $\Gamma$ is not as in~$(A)$. As noted above, we must have 
$m=2t$ and $|E|=2^{2t+1}$, from which it follows that $|\Z(E)|=2$. It remains to show that $(\Gamma,G)$ is one of $(\Gamma_t^+,G_t^+)$ or $(\Gamma_t^-,G_t^-)$. Since $1\neq N\leq \Z(E)$, we have $N=\Z(E)$. Furthermore, from~$(\ref{eq:fact0})$ we obtain 
\begin{equation}\label{eq:fact}
E=\Z(E)E_v^{a^t}E_v.
\end{equation}

As $E/\Z(E)$ is an elementary abelian $2$-group and $E$ is non-abelian, we have $E^2=[E,E]=\Z(E)$. Write $\Z(E)=\langle z\rangle$. 

\smallskip

\noindent\textsc{Claim~$4$. }
\begin{eqnarray*}
[x_i,x_j]&=&\left\{
\begin{array}{ccl}
1&&\textrm{if }|j-i|\leq t-1,\\
z&&\textrm{if }|j-i|=t.
\end{array}
\right.
\end{eqnarray*}

\noindent If $0\leq j-i\leq t-1$, then $x_i$ and $x_j$ are both contained in $E_t^{a^i}=\langle x_i,\ldots,x_{i+t-1}\rangle$ which is abelian and hence they commute. Suppose that $[x_0,x_t]=1$. It follows that $x_t$ commutes with $E_v=\langle x_0,\ldots,x_{t-1}\rangle$ and with $E_v^{a^t}=\langle x_t,\ldots,x_{2t-1}\rangle$ and hence is central in $E$, which is contradiction. Hence, $[x_0,x_t]\neq 1$ and, for each $i$, $[x_i,x_{i+t}]=[x_0,x_t]^{a^i}\neq 1$. Since $[E,E]=\Z(E)$, it follows that $[x_i,x_{i+t}]=z$.~$_\blacksquare$

\smallskip

\noindent\textsc{Claim~$5$. }Replacing $a$ by an element in the coset $E_va$ if necessary, we have
\begin{eqnarray*}
x_i^a&=&x_{i+1}\;(\textrm{for }0\leq i\leq 2t-2), \textrm{ }x_{2t-1}^a=x_0, \textrm{ and } v^a \textrm{ is adjacent to } v.
\end{eqnarray*} 

\noindent Since $|G^+:E|=2t$, we have that $a^{2t}\in E$. It follows that $x_{2t}=x_0^{a^{2t}}=x_0[x_0,a^{2t}]=x_0z^\varepsilon$ for some $\varepsilon\in\{0,1\}$. If $\varepsilon=0$, then there is nothing to prove. We assume that $\varepsilon=1$ and let $a'=x_{t-1}a$. By Claim~$4$, we get that, for $0\leq i\leq 2t-2$, $x_i$ commutes with $x_{t-1}$ and hence $x_i^{a'}=x_{i+1}$. Moreover $x_{2t-1}^{a'}=(x_{2t-1}z)^a=x_{2t}z^a=(x_{0}z)z=x_{0}$ and hence the claim is proved replacing $a$ by $a'$.~$_\blacksquare$

\smallskip

From Claim~$5$, it follows that, for every $i$, $x_i^{a^{2t}}=x_i$. Since $a^{2t}\in E$, we have
\begin{equation}\label{eq:2part2}
a^{2t}\in\Z(E).
\end{equation}
Moreover, since $x_{i+2t}=x_i^{a^{2t}}=x_i$, from now on the index $i$ of $x_i$ will be taken modulo $2t$. Let $b$ be an element of $G_v\setminus E_v$. 

\smallskip

\noindent\textsc{Claim~$6$. }Replacing $b$ by an element in the coset $bE_v$ if necessary, we have $x_{i}^b=x_{t-1-i}$ for every $i$.

\noindent Since $b$ normalises $E_v$ and $x_{t-1}\in E_v$, we have $x_{t-1}^b=x_0^{\varepsilon_0}\cdots x_{t-1}^{\varepsilon_{t-1}}$ for some $\varepsilon_i\in\{0,1\}$. As $G/E\cong \D_m$, we get $(a^{1-t})^b=a^{t-1}y$ for some $y\in E$. Hence
\begin{eqnarray*}
x_0^b&=&x_{t-1}^{a^{1-t}b}=x_{t-1}^{ba^{t-1}y}=(x_0^{\varepsilon_0}\cdots x_{t-1}^{\varepsilon_{t-1}})^{a^{t-1}y}=(x_{t-1}^{\varepsilon_0}\cdots x_{2t-2}^{\varepsilon_{t-1}})^y\\
&=&x_{t-1}^{\varepsilon_0}\cdots x_{2t-2}^{\varepsilon_{t-1}}[x_{t-1}^{\varepsilon_0}\cdots x_{2t-2}^{\varepsilon_{t-1}},y]=x_{t-1}^{\varepsilon_0}\cdots x_{2t-2}^{\varepsilon_{t-1}}z^{\varepsilon}
\end{eqnarray*}
for some $\varepsilon\in\{0,1\}$. Since $x_0^b\in E_v=\langle x_0,\ldots,x_{t-1}\rangle$, from the previous equation we obtain $\varepsilon_1=\cdots=\varepsilon_{t-1}=\varepsilon=0$ and $\varepsilon_0=1$, that is, $x_0^b=x_{t-1}$ and $x_{t-1}^b=x_0$. Fix $i$ in  $\{0,\ldots,t-1\}$ and let $y\in E$ with $(a^{1-t+i})^b=a^{t-1-i}y$. We get
$x_i^b=x_{t-1}^{a^{1-t+i}b}=x_{t-1}^{ba^{t-1-i}y}=x_0^{a^{t-1-i}y}=x_{t-1-i}^y=x_{t-1-i}z^\varepsilon$ for some $\varepsilon\in\{0,1\}$. Since $x_{t-1-i},x_i^b\in E_v$, we have $\varepsilon=0$ and hence $x_i^b=x_{t-1-i}$. 
 As $G/E\cong \D_{2t}$, we get $(a^t)^b=a^ty$ for some $y\in E$. Using~$(\ref{eq:fact})$, write $y=z^\varepsilon y_1y_2$ with $\varepsilon\in\{0,1\}$, $y_1\in E_v^{a^t}$ and $y_2\in E_v$. Let $b'=by_2$. As $y_2\in E_v$, for $i\in\{0,\ldots,t-1\}$, we have $x_{i}^{b'}=x_{t-i-1}^{y_2}=x_{t-i-1}$. Fix $i$ in $\{0,\ldots,t-1\}$, we obtain $x_{t+i}^{b'}=x_i^{a^tby_2}=x_{i}^{ba^tyy_2}=x_{t-1-i}^{a^tz^\varepsilon y_1}=x_{2t-1-i}^{z^\varepsilon y_1}=x_{2t-1-i}$, where in the last equality we used that $y_1,x_{2t-1-i}\in E_v^{a^t}$ and that $E_v^{a^t}$ is abelian. 
Hence the claim is proved replacing $b$ by $b'$.~$_\blacksquare$

\smallskip

From Claim~$6$, it follows that $b^2$ centralises $E$. Since $b\in G_v\setminus E_v$, $|G_v:E_v|=2$ and $E_v\cap \Z(E)=1$, we have
\begin{equation}\label{eq:5part2}
b^2=1.
\end{equation}

As $G/E\cong \D_{2t}$, we get $aa^b\in E$. From Claims~$5$ and~$6$, we have $x_i^{aa^b}=x_{i+1}^{bab}=x_{t-2-i}^{ab}=x_{t-1-i}^b=x_i$. Therefore $aa^b\in \Z(E)$ and $a^b=a^{-1}z^{\varepsilon}$ for some $\varepsilon\in\{0,1\}$. Since $G$ is arc-transitive on $\Gamma$ and $v^a$ is adjacent to $v$, there exists $g\in G$ such that $(v,v^a)^g=(v^a,v)$. It follows that $ag\in G_v$ and $ga^{-1}\in G_v$ which yields $ag\in G_v\cap aG_va$ and hence $G_v\cap aG_va=(E_v\cup E_vb)\cap aG_va\neq \emptyset$. Since $a$ acts as a rotation of order $2t\geq 3$ on $\vGa/E$, we have that the elements in $a^{-2}E=a^{-1}Ea^{-1}$ act fixed-point-freely on $\V\Gamma$. Therefore $a^{-1}Ea^{-1}$ intersects $G_v$ trivially and hence $E_v\cap aG_va=\emptyset$. This gives $E_vb\cap aG_va\neq \emptyset$. Let $yb\in E_vb\cap aG_va$ with $y=x_0^{\varepsilon_0}\cdots x_{t-1}^{\varepsilon_{t-1}}\in E_v$. As $a^b=a^{-1}z^\varepsilon$, we obtain  $a^{-1}yba^{-1}=y^aa^{-1}ba^{-1}=y^abz^\varepsilon\in G_v$. Since
  $G
 _v=E_v\cup E_vb$, it follows that $$y^az^\varepsilon=x_1^{\varepsilon_0}\cdots x_{t-1}^{\varepsilon_{t-2}}x_t^{\varepsilon_{t-1}}z^{\varepsilon}\in E_v$$
and hence $x_t^{\varepsilon_{t-1}}z^{\varepsilon}\in E_v\cap E_v^{a^t}\Z(E)$.  Now~$(\ref{eq:fact})$ yields $E_v\cap E_v^{a^t}\Z(E)=1$ and hence $\varepsilon=0$. Therefore
\begin{equation}\label{eq:6part2}
a^b=a^{-1}.
\end{equation}

Recalling the definitions of $(\Gamma_t^+,G_t^+)$ and $(\Gamma_t^-,G_t^-)$ from Section~\ref{ss:Gamma}, it follows from Claims~$4$,~$5$ and~$6$ and from~$(\ref{eq:2part2})$,~$(\ref{eq:5part2})$ and~$(\ref{eq:6part2})$ that either $(\Gamma,G)=(\Gamma_t^+,G_t^+)$ (if $a^{2t}=1$) or $(\Gamma,G)=(\Gamma_t^-,G_t^-)$ (if $a^{2t}=z$).
\end{proof}

\section{Case where $G$ has a non-abelian minimal normal subgroup}\label{eaAS}

Our main tool in this section is the following observation, which follows from~\cite[Theorem]{D}.

\begin{lemma}\label{lemma:ea}
Let $(\Gamma,G)$ be a locally-$\D_4$ pair and let $v\in \V\Gamma$. Then $G_v$ is a $2$-group and contains a subgroup $P$ of index $2$ and of nilpotency class at most $2$. Moreover, $P$ contains an elementary abelian $2$-subgroup $E$ of order $2^m$ with $|P|\leq 2^{\lfloor 3m/2\rfloor}$ and, in particular, $|G_v|\leq 2|E|^{3/2}$.
\end{lemma}

As Lemma~\ref{lemma:ea} indicates, the order of  $G_v$ in a locally-$\D_4$ pair can be bounded from above by a function of the order of a  maximal elementary abelian $2$-group of $G_v$. This motivates the introduction of the following definition.

\begin{definition}\label{def:eg}
The \emph{$2$-rank} $r_G$ of a finite group $G$ is the minimal number of generators of an elementary abelian $2$-subgroup of maximal order of $G$. We denote by $e_G$  the number $2^{r_G}$, that is, $e_G$ is the order of an elementary abelian $2$-subgroup of $G$ of maximal order.
\end{definition}

Before proving the main results of this section we need the following lemma on the $2$-rank of a wreath product.

\begin{lemma}\label{lemma:wp}
Let $W=H\wr_{\Delta} K$. If $|H|$ is even, then $e_W\leq e_{H}^{|\Delta|}$.
\end{lemma}
\begin{proof}
Consider $W$ as a permutation group on $\Omega=H\times \Delta$ where $H$ acts on the set $H$ by right multiplication. We identify the system of imprimitivity $\Sigma=\{H\times \{\delta\}\mid \delta\in \Delta\}$ with $\Delta$, where the block $H\times\{\delta\}\in\Sigma$ is identified with $\delta\in \Delta$. In particular, we say that a subgroup of $W$ is transitive on $\Delta$ if it is transitive on $\Sigma$. Note that the kernel of the action of $W$ on $\Delta$ is $B=H^{\Delta}$.

Let $E$ be an elementary abelian 2-subgroup of $W$ and let $\mathcal{O}_1,\ldots,\mathcal{O}_r$ be the orbits of $E$ on $\Delta$. Note that
$E\leq (H\wr_{\mathcal{O}_1}K)\times \cdots \times
(H\wr_{\mathcal{O}_r}K)$. We argue by induction on $r$. Assume first that $r=1$ (that is, $E$ acts transitively on $\Delta$). Let $f$ be an element of $E\cap B$. Let $\delta_1,\delta_2\in\Delta$ and let $e=\sigma g\in E$ with $\delta_1^\sigma=\delta_2$ where $g\in B$, $\sigma\in K$. Since $E$ is abelian, it follows that $f(\delta_2)=f(\delta_1^\sigma)=f^e(\delta_1)=f(\delta_1)$ and hence $f$ is a constant function of $B$. This yields that $E\cap B$ is isomorphic to an elementary abelian subgroup of $H$ and hence $|E\cap B|\leq e_{H}$. As $E$ is abelian and transitive on $\Delta$, the group $E/(E\cap B)$ acts regularly on $\Delta$ and hence $|E:E\cap B|=|\Delta|$. It follows that $|E|=|E\cap B||E:E\cap B|\leq e_{H}|\Delta|$. Since $|H|$ is even, we have $e_{H}\geq 2$ and hence $e_H
 |\Delta| \leq e_{H}^{|\Delta|}$. Assume now that $r\geq 2$.  Since $E\leq
(H\wr_{\mathcal{O}_1}K)\times \cdots \times (H\wr_{\mathcal{O}_r}K)$, using the induction hypothesis, we obtain
$$|E|\leq \prod_{i=1}^r|H\mathrm{wr}_{\mathcal{O}_i}K|\leq \prod_{i=1}^re_{H}^{|\mathcal{O}_i|}=e_{H}^{\sum_i|\mathcal{O}_i|}=e_H^{|\Delta|},$$ completing the induction and the proof.
\end{proof}

The following technical theorem is the key ingredient in the proof of our main result for this section. The proof depends heavily on the classification of finite simple groups and includes a very long case-by-case analysis, hence we defer it to Section~\ref{TechnicalProof}.

\begin{theorem}\label{theo:key}Let $T$ be a non-abelian simple group and $l\geq 1$. Write $l=2^{l_e}l_o$ with $l_o$ odd, $|T|=2^to$ with $o$ odd and $e=e_{\Aut(T)}$. Then either $l_oo^l>6le^{3l/2}\log_2(e)$  or one of the following holds:
\begin{description}
\item[$(i)$]$l\in\{1,2,3,4\}$ and $T=\Alt(5)$ or $\Alt(6)$;
\item[$(ii)$]$l\in\{1,2,4\}$ and $T=A_1(8)$ or $A_2(2)$;
\item[$(iii)$]$l\in\{1,2\}$ and $T=A_1(2^f)$ (with $f\in\{4,5,6\}$), $\Alt(8)$, $A_2(4)$, $B_2(4)$, $B_2(3)$ or $B_3(2)$; 
\item[$(iv)$]$l=1$ and  $T=A_1(2^f)$ (with $f\in\{7,\ldots,12\}$), $M_{12}$, $M_{22}$, $\Alt(7)$, $A_1(11)$, $A_1(13)$, $A_1(25)$, $A_4(2)$, $A_5(2)$, $A_2(3)$, $B_2(8)$, $B_2(16)$, $B_2(32)$,  $B_3(4)$, $B_4(2)$, $D_4(2)$, ${^2A}_2(3)$, ${^2A}_5(2)$ or  $^2{D}_4(2)$.
\end{description}
\end{theorem}

In the rest of this section, we use a few well-known facts about a group $G$ with a unique minimal normal subgroup $N$ (see \cite[Section 4.3]{DixMor}). In particular, if $N$ is non-abelian, we have $N\cong T^l$ for some non-abelian simple group $T$ and for some $l\geq 1$, and $G$ acts transitively on the $l$ simple direct summands of $N$ by conjugation. Moreover, as $\C_G(N)=1$, the group $G$ can be embedded in $\Aut(N)\cong \Aut(T)\wr\Sym(l)$. With some more computations, we get the following corollary to Theorem~\ref{theo:key}.

\begin{corollary}\label{prop:end}
Let $(\Gamma,G)$ be a locally-$\D_4$ pair. Assume that $G$ has a unique minimal normal subgroup $N$ and that $N$ is isomorphic to $T^l$ where $T$ and $l$ are as in $(i)\ldots(iv)$ of the conclusion of Theorem~\ref{theo:key}. Then either $|\V\Gamma|> 2|G_v|\log_2(|G_v|/2)$ or $(\Gamma,G)$ is one of the pairs in rows $(i)$, $(ii)$, $(iii)$, $(iii)a$, $(iii)b$ or $(iv)$ of Table~\ref{tb:nsoluble}.
\end{corollary}
\begin{proof}
Note that there are only finitely many pairs $T,l$ appearing in $(i)\ldots(iv)$ of the conclusion of Theorem~\ref{theo:key}. Therefore this corollary can be proved with the help of a computer. Nevertheless some of the computations involved are non-trivial, hence we give some details on how the result is obtained using~\texttt{Magma}. Assume $|\V\Gamma|\leq 2|G_v|\log_2(|G_v|/2)$. Given $T$ and $l$ as in the statement, the number of groups $G$ with a unique minimal normal subgroup isomorphic to $T^l$ is very limited (as $|\Out(T)|$ is small). Fix such a $G$ and let $S$ be a Sylow $2$-subgroup of $G$. Let $L$ be the set of divisors $n$ of $|S|$ such that $|G|/n\leq 2n\log_2(n/2)$. These are our candidates for $|G_v|$. In most cases (but not always), $L$ is either the empty set or the set containing only $|S|$. In particular, the number of subgroups $Q$ of $S$ such that $|G:Q|\leq 2|Q|\log_2(|Q|/2)$ is always very small. These are our candidates for $G_v$. For each such $Q$, we 
 check whether $Q$ has a maximal subgroup $P$ of nilpotency class at most $2$. If this is not the case, then, by Lemma~\ref{lemma:ea}, we can discard $Q$. Finally, for the remaining $Q$'s, we construct the permutation representation of $G$ on the right cosets of $Q$ and check whether there exists a self-paired suborbit of size $4$ giving rise to a connected locally-$\D_4$ pair.
\end{proof}

We are now ready to prove the main result of this section.

\begin{theorem}\label{thm:mainss}
Let $(\Gamma,G)$ be a locally-$\D_4$ pair. Assume that $G$ has a non-abelian minimal normal subgroup. Then either $|\V\Gamma|>2|G_v|\log_2(|G_v|/2)$ or  $(\Gamma,G)$ is one of the pairs in Table~\ref{tb:nsoluble}.
\end{theorem}
\begin{proof}
Let $(\Gamma,G)$ be a counter-example to Theorem~\ref{thm:mainss}, minimal with respect to $|\V\Gamma|$. Let $v$ be a vertex of $\Gamma$ and let $N$ be a non-abelian minimal normal subgroup of $G$. Assume that $G$ has a minimal normal subgroups $M\neq N$. In particular, $NM=N\times M$. Let $K$ be the kernel of the action of $G$ on the vertices of $\Gamma/M$. Suppose that $N\leq K$. Then $v^N\subseteq v^K\subseteq v^M$. Let $n\in N$ be a non-identity element of odd order. We have $v^n\in v^M$ and hence $v^n=v^m$ for some $m\in M$. This gives $nm^{-1}\in G_v$. Since $|nm^{-1}|=|n||m|$ is not a power of 2 and $G_v$ is a $2$-group, this is a contradiction which yields $N\nleq K$. 

By minimality of $N$, we obtain $N\cap K=1$ and hence the group $N\cong NK/K$ acts faithfully as a group of automorphisms on $\Gamma/M$. Since a connected graph of valency at most $2$ has soluble automorphism group, it follows that $\Gamma/M$ has valency $4$ and hence $K_v=1$, $K=M$ and $(\Gamma/M,G/M)$ is locally-$\D_4$. By the minimality of $(\Gamma,G)$, we have that either $|\V(\Gamma/M)|>2|G_v|\log_2(|G_v|/2)$ or $(\Gamma/M,G/M)$ is one of the pairs in Table~\ref{tb:nsoluble}. In the former case, $|\V\Gamma|>|\V(\Gamma/M)|$ and the theorem follows. In the latter case, the theorem follows from Lemma~\ref{thm:exp}.

From now on, we may assume that $N$ is the unique minimal normal subgroup of $G$. Let $T$ be a non-abelian simple group and $l\geq 1$ with $N\cong T^l$.
Write $e=e_{\Aut(T)}$, $l=2^{l_e}l_o$ with $l_o$ odd and $|T|=2^to$ with $o$ odd. As $G_v$ is a $2$-group, we have $|\V\Gamma|\geq |G:S|$ where $S$ is a Sylow $2$-subgroup of $G$.

Since $N$ is the unique minimal normal subgroup of $G$, the group $G$ is isomorphic to a subgroup of $\Aut(N)\cong \Aut(T)\wr\Sym(l)$ and $|G|\geq l|N|$, from which it follows that $|\V\Gamma|\geq |G:S|\geq l_oo^l$. Moreover, by Lemma~\ref{lemma:wp}, we have $e_{G_v}\leq e_{\Aut(N)}\leq e^l$. Together with Lemma~\ref{lemma:ea}, this yields $|G_v|\leq 2e^{3l/2}$.  

If $l_oo^l>6le^{3l/2}\log_2(e)$, then $|\V\Gamma|\geq l_oo^l>6le^{3l/2}\log_2(e)=4e^{3l/2}\log_2(e^{3l/2})\geq 2|G_v|\log_2(|G_v|/2)$ and the theorem follows. Therefore we may assume that $l_oo^l\leq 6le^{3l/2}\log_2(e)$ and now the conclusion follows from Theorem~\ref{theo:key} and Corollary~\ref{prop:end}.
\end{proof}

\section{Proof of Theorem~\ref{thm:main}}\label{MainProof}

\noindent\emph{Proof of Theorem~\ref{thm:main}. }Let $(\Gamma,G)$ be locally-$\D_4$ and $N$  a minimal normal subgroup of $G$. We argue by induction on $|\V\Gamma|$. If $N$ is non-abelian, then from Theorem~\ref{thm:mainss} we get that either part~$(B)$ or~$(C)$ holds for $(\Gamma,G)$. Furthermore if $(C)$ does hold for $(\Gamma,G)$, then the inequality is strict. Hence we may assume that $N$ is abelian. If $\Gamma/N$ has valency at most $2$, then it follows from Theorem~\ref{thm:mainbasic} that one of~$(A)$,~$(B)$ or~$(C)$ holds. Furthermore, if~$(C)$ does hold, then the inequality is strict.  Hence we may assume that $\Gamma/N$ is $4$-valent. In particular, $N$ acts semiregularly on $\V\Gamma$, $(\Gamma/N,G/N)$ is locally-$\D_4$ and the vertex-stabiliser in $G/N$ is isomorphic to $G_v$. By induction, it follows that $(\Gamma/N,G/N)$ satisfies one of~$(A)$,~$(B)$ or~$(C)$.

If~$(C)$ holds for $(\Gamma/N,G/N)$, then $|\V\Gamma|=|N||\V(\Gamma/N)|\geq|N|(2|G_v|\log_2(|G_v|/2))>2|G_v|\log_2(|G_v|/2)$ and~$(C)$ holds for $(\Gamma, G)$ with the inequality being strict. If~$(A)$ holds for $(\Gamma/N,G/N)$, then it follows from Theorem~\ref{WreathCover} (and the subsequent remark)  that $(\Gamma,G)$ satisfies~$(A)$ or~$(C)$. Moreover,
 if the pair $(\Gamma,G)$ meets the bound in~$(C)$ and $\Gamma$ is not as in~$(A)$, then $(\Gamma,G)\cong (\Gamma_t^{\pm},G_t^{\pm})$ for some $t\ge 2$. Suppose now that~$(B)$ holds for $(\Gamma/N,G/N)$, that is, $(\Gamma/N,G/N)$ is one of the pairs in Tables~\ref{tb:soluble} and~\ref{tb:nsoluble}. From Lemmas~\ref{thm:exp2} and~\ref{thm:exp} we obtain that~$(B)$ or~$(C)$ holds for $(\Gamma,G)$. Furthermore, if $(C)$ does hold, then the inequality is strict.~$\qed$

\section{Proof of Theorem~\ref{theo:key}}\label{TechnicalProof}
We now return to the proof of Theorem~\ref{theo:key}, which we skipped earlier. The first step is to collect information about the $2$-ranks of non-abelian simple groups, starting with sporadic groups. Table~\ref{spor} gives $e_T$ when $T$ is a sporadic simple group. This table was obtained using~\cite{Law} when $T\in\{B,M\}$ and ~\cite[Table~$5.6.1$, page~$303$]{GLS} in the rest of the cases.

 \begin{table}[!h]
\begin{center}
 \begin{tabular}{|c|ccccccccc|}\hline
$T$  & $M_{11}$&$M_{12}$&$M_{22}$&$M_{23}$&$M_{24}$&$J_1$&$J_2$& $J_3$&$J_4$\\
$e_T$&$2^2$    & $2^3$ & $2^4$ &$2^4$   &$2^6$  & $2^3$&$2^4$& $2^4$&$2^{11}$\\\hline\hline
$T$&$Co_1$&$Co_2$&$Co_3$&$Suz$& $Fi_{22}$& $Fi_{23}$& $Fi_{24}'$ & $HS$ & $McL$\\
$e_T$& $2^{11}$&$2^{10}$ & $2^4$&$2^6$&$2^{10}$&$2^{11}$&$2^{11}$&$2^4$&$2^4$\\\hline\hline
$T$&$He$ & $HN$ & $Th$& $B$ &  $M$ & $O'N$ & $Ly$&$Ru$&\\
 $e_T$& $2^6$ &
 $2^6$&$2^5$&$2^{14}$&$2^{15}$&$2^3$&$2^4$&$2^6$&\\\hline
 \end{tabular}
\caption{$e_T$ for $T$ a sporadic simple group}\label{spor}
\end{center}
 \end{table}

The next step is to compute $e_{\Alt(n)}$ and $e_{\Sym(n)}$.

\begin{lemma}\label{e:alt}
Let $n$ be a positive integer and write $n=4m+r$ with $0\leq r\leq 3$. If $r=0$ or $1$, then $e_{\Sym(n)}=e_{\Alt(n)}=2^{2m}$. If $r=2$ or $3$, then $e_{\Sym(n)}=2^{2m+1}$ and $e_{\Alt(n)}=2^{2m}$. 
\end{lemma}

\begin{proof}
If $r$ is odd, then a Sylow $2$-subgroup of $\Sym(n)$ fixes some point of $\{1,\ldots,n\}$ and hence is conjugate to a Sylow $2$-subgroup of $\Sym(n-1)$. In particular, we may assume that $r$ is even without loss of generality. Define
\begin{eqnarray*}
E_0&=&
\langle (1,2)(3,4),(1,3)(2,4),\ldots,(m-3,m-2)(m-1,m),(m-3,m-1)(m-2,m)\rangle
\end{eqnarray*}
if $r=0$, and 
\begin{eqnarray*}
E_0&=&
\langle (1,2)(3,4),(1,3)(2,4),\ldots,(m-3,m-2)(m-1,m),(m-3,m-1)(m-2,m),\\
&&(m+1,m+2)\rangle
\end{eqnarray*}
if $r=2$. The group $E_0$ is an elementary abelian $2$-subgroup of $\Sym(n)$ of order $2^{n/2}$ and hence $e_{\Sym(n)}\geq |E_0|=2^{n/2}$. Let $E$ be an elementary abelian $2$-subgroup of maximal order in $\Sym(n)$ and let $\mathcal{O}_1,\ldots,\mathcal{O}_k$ be the orbits of $E$ on $\{1,\ldots,n\}$. As $E$ is abelian, the action of $E$ on $\mathcal{O}_i$ is regular for every $i$ and hence $|E|\leq \prod_{i=1}^{k}|\mathcal{O}_i|$. Note that, if $a$ is a power of $2$, then $a\leq 2^{a/2}$ with equality if and only if $a\in\{2,4\}$. Using the  maximality of $|E|$, we have $$2^{n/2}=|E_0|\leq |E|\leq \prod_{i=1}^k|\mathcal{O}_i|\leq\prod_{i=1}^k2^{|\mathcal{O}_i|/2}=2^{\sum_{i=1}^k|\mathcal{O}_i|/2}=2^{n/2}.$$ This shows that $e_{\Sym(n)}=2^{n/2}$. Moreover, this also shows that the order of an elementary abelian $2$-subgroup $E$ of $\Sym(n)$ is $2^{n/2}$ if and only if $E$ is the direct product of the permutation groups induced by $E$ on each of its orbits and each such or
 bit has size $2$ or $4$. 

If $r=0$, then $E_0\leq \Alt(n)$ and hence $e_{\Alt(n)}=e_{\Sym(n)}$. Finally, if $r=2$, then, by the previous paragraph, an elementary abelian subgroup $E$ of $\Sym(n)$ of order $2^{n/2}$ has an orbit of size $2$ and contains a transposition. In particular, $e_{\Alt(n)}<e_{\Sym(n)}$. Since $|E_0\cap \Alt(n)|=2^{2m}$, we obtain $e_{\Alt(n)}=2^{2m}$. 
\end{proof}
 
If $r=0$, then $E_0\leq \Alt(n)$ and hence $e_{\Alt(n)}=e_{\Sym(n)}$. Finally, if $r=2$, then, by the previous paragraph, an elementary abelian subgroup $E$ of $\Sym(n)$ of order $2^{n/2}$ has an orbit of size $2$ and contains a transposition. In particular, $e_{\Alt(n)}<e_{\Sym(n)}$. Since $|E_0\cap \Alt(n)|=2^{2m}$, we obtain $e_{\Alt(n)}=2^{2m}$.

Although there is an extensive literature on the Sylow $2$-subgroups of simple groups of Lie type $T$, we were unable to find an explicit reference for $e_T$ in this case. We wish to thank B.~Stellmacher for an enlightening conversation with the second author which inspired the proof of the following technical lemma, where we compute $e_{\PSL(n,q)}$ when $q$ is odd.

\begin{lemma}\label{e:PSL}
Let $n\geq 1$ and let $q$ be odd. Then $e_{\mathrm{PGL}(n,q)}\leq 2^n$. Also, if $n$ is odd, then $e_{\mathrm{PSL}(n,q)}\leq 2^{n-1}$.
\end{lemma}  

\begin{proof}If $n=1$, then the result is clear. If $n=2$, then from the description of the subgroups of $\PSL(2,q)$ in~\cite[\S~$6$, Theorem~$6.25$]{Sz1}, we obtain $e_{\PSL(2,q)}=4$ for every odd $q$. Also, as $\mathrm{PGL}(2,q)$ is isomorphic to a subgroup of $\mathrm{PSL}(2,q^2)$ (see~\cite[\S~$6$, Theorem~$6.25$~$(d)$]{Sz1}), we get $e_{\mathrm{PGL}(2,q)}=4$. Thence, from now on, we may assume $n>2$.

Let $V$ be an $n$-dimensional vector space over the field with $q$ elements $\mathbb{F}_q$. Write $G=\GL(V)$, $S=\SL(V)$, $\overline{G}=G/\Z(G)$ and $\overline{S}=S/\Z(S)$  where $\Z(G)$ denotes the centre of $G$ and $\Z(S)=S\cap \Z(G)$. Moreover, let $\overline{A}$ be an elementary abelian $2$-subgroup of $\overline{G}$ (respectively $\overline{S}$) and $A$ a $2$-subgroup of $G$ (respectively $S$) such that $\overline{A}=A\Z(G)/\Z(G)$ (respectively $\overline{A}=A\Z(S)/\Z(S)$). We prove two preliminary claims.

\smallskip

\noindent\textsc{Claim~1. }$|[A,A]|\leq 2$.

\noindent Let $a$ and $b$ be in $A$. Since $\overline{A}$ is an elementary abelian $2$-group, we get $a^2\in \Z(G)$ and $[A,A]\leq \Z(G)$. In particular, as $\Z(G)$ is cyclic, the group $[A,A]$ is cyclic. From one of the basic commutator identities, we obtain $1=[a^2,b]=[a,b]^a[a,b]=[a,b]^2$. This proves that $[A,A]$ has exponent at most $2$ and therefore $|[A,A]|\leq 2$.~$_\blacksquare$  

\smallskip

Given a subgroup $H$ of $G$ we write $[V,H]=\{v-v^h\mid v\in V,h\in H\}$.

\smallskip

\noindent\textsc{Claim~2. }Let $a$ and $b$ be in $A$ such that $a^2=1$ and $z=[a,b]\neq 1$. We have $[V,\langle a\rangle]=\{v\in V\mid v^a=-v\}=\C_V(az)$ and $[V,\langle az\rangle]=\{v\in V\mid v^{az}=-v\}=\C_V(a)$. Also, $V=[V,\langle a\rangle]\oplus [V,\langle a^b\rangle]$ and $n/2=\dim_{\mathbb{F}_q}[V,\langle a\rangle]$. In particular, $n$ is even.

\noindent By Claim~$1$, $z^2=1$. Also, as $\overline{A}$ is elementary abelian, $z\in \Z(G)$. Therefore the element $z$ acts on $V$ by the multiplication by $-1$.  Let $w$ be in $[V,\langle a\rangle]$, that is, $w=v-v^a$ for some $v\in V$. We have $w^a=(v-v^a)^a=v^a-v^{a^2}=v^a-v=-w$. Conversely, if $w^a=-w$, then $w=w/2-(w/2)^a\in [V,\langle a\rangle]$. This gives $[V,\langle a\rangle]=\{v\in V\mid v^a=-v\}=\C_V(az)$. Similarly, as $az=a^b$ is conjugate to $a$, we get $[V,\langle az\rangle]=\{v\in V\mid v^{az}=-v\}=\C_V(a)$.

Since the order of $A$ and $V$ are coprime, from~\cite[\S~$1$, page~$7$]{Sz}, we have $V=[V,\langle a\rangle]\oplus \C_V(a)$. From the previous paragraph, it follows that $V=[V,\langle a\rangle]\oplus [V,\langle a^b\rangle]$. Finally, as $\C_V(a)^b=\C_V(a^b)=\C_V(az)$, we get $\dim_{\mathbb{F}_q}\C_V(a)=\dim_{\mathbb{F}_q}[V,\langle a\rangle]$. In particular, $n/2=\dim_{\mathbb{F}_q}[V,\langle a\rangle]$ and $n$ is even.~$_\blacksquare$

\smallskip
Assume $A$ is abelian. Since $\overline{A}$ is a quotient of the abelian group $A$, we have $r_{\overline{A}}\leq r_{\Omega_1(A)}$ where $\Omega_1(A)=\{a\in A\mid a^2=1\}$. Now, $\Omega_1(A)$ is an elementary abelian $2$-subgroup of $G$ (respectively $S$). Hence by Lemma~\ref{PabloLemma}, we have $r_{\Omega_1(A)}\leq n$ (respectively $r_{\Omega_1(A)}\leq n-1$) and the bound for $r_{\overline{A}}$ is proved.

Assume $\Omega_1(A)\nleq \Z(A)$ and $A$ is non-abelian. Since not every element of order $2$ of $A$ is contained in the centre of $A$, there exist $a,b \in A$ such that $a^2=1$ and $z=[a,b]\neq 1$. By Claim~$1$, $z$ has order $2$ and $[a,A]=[b,A]=\langle z\rangle$. In particular, $|A:\C_A(a)|=|A:\C_A(b)|=2$.  Writing $B=\C_A(a)$ and $B_1=B\cap \C_A(b)=\C_A(\langle a,b\rangle)$, we obtain $|A:B_1|=4$, $|B:B_1|=2$ and $B=B_1\langle az\rangle$. From Claim~$2$, we have that $[V,\langle a\rangle]=\{v\in V\mid v^a=-v\}$ and $[V,\langle a\rangle]^b=[V,\langle a^b\rangle]=[V,\langle az\rangle]=\C_V(a)$. Thus, $A$ does not normalise $[V,\langle a\rangle]$. Since $B$ normalises $[V,\langle a\rangle]$ and $|A:B|=2$, we conclude that $B=N_A([V,\langle a\rangle])$.

Set $C=\C_B([V,\langle a\rangle])$. From Claim~$2$, we have that $az\in C$ and that $$C\cap C^b=\C_B([V,\langle a\rangle])\cap \C_B([V,\langle az\rangle])=\C_B([V,\langle a\rangle]+[V,\langle az\rangle])=\C_B(V)=1.$$ 

Since $[A,A]=\langle z\rangle$ and $A/\langle z\rangle$ is abelian, the
map $x\mapsto [x,b]$ is a group homomorphism from $A$ to $\langle z\rangle$. It follows that $\C_A(b)$
has index 2 in $A$ and hence $\C_C(b)$ has index 2 in $C$ and is contained in $C\cap
C^b$. As $C\cap C^b=1$, we obtain $C=\langle az\rangle$. In particular, $B_1$ acts faithfully on $[V,\langle a\rangle]$.

Set $B_0=B_1\cap \Z(\GL([V,\langle a\rangle]))$. We have $A\cap \Z(G)\leq B_0$ and $[B_0,b]=1$. This shows that $[V,\langle a\rangle]$ and $[V,\langle a\rangle^b]$ are isomorphic $\mathbb{F}_qB_0$-modules. Since $V=[V,\langle a\rangle]\oplus [V,\langle a\rangle^b]$, we get $B_0\leq \Z(G)$ and $B_0=A\cap \Z(G)$. Finally, by induction on $\dim_{\mathbb{F}_q}V$, we get 
$$|B_1/B_0|=|B_1/(A\cap \Z(G))|\leq 2^{\dim_{\mathbb{F}_q}[V,\langle a\rangle]}=2^{n/2}.$$
Since $|A:B_1|=4$, this implies $|\overline{A}|=|A/(A\cap \Z(G))|\leq 4\cdot 2^{n/2}\leq 2^{n}$ (where in the last inequality, we used the fact that $n\geq 4$).

Finally, assume that $\Omega_1(A)\leq \Z(A)$ and $A$ is non-abelian. Write $\Omega_1(A)=Z\times A_0$, where $Z=\Omega_1(A)\cap \Z(G)$. Then, by Claim~$1$, we have that $Z=[A,A]$ and $|Z|=2$. Recall that $A/(A\cap \Z(G))$ is elementary abelian. Note that, as $A_0\leq\Omega_1(A)\leq Z(A)$, the group $A_0$ is normal in $A$. If $aA_0$ lies in $A/A_0$ and $aA_0$ has order $2$, then $a^2\in A_0\cap \Z(G)=A_0\cap Z=1$. This says that the elements of order $2$ in  $A/A_0$ are the elements in $\Omega_1(A)/A_0\cong Z$, that is, $A/A_0$ contains a unique element of order  $2$. Thus, $A/A_0$ is the quaternion group of order $8$. In particular, $|\overline{A}|=4|A_0|$. If $A_0=1$, then $|\overline{A}|$ has order $4$ and $4\leq 2^{n-1}$ (recall that $n\geq 3$). Thus, we may assume that $A_0\neq 1$. Since the order of $V$ is coprime to $2$, by Lemma~\ref{PabloLemma}, the action of the elementary abelian $2$-group $\Omega_1(A)$ on $V$ can be diagonalized. It follows that there exists a 
 subgroup $R$ of index $2$ in $\Omega_1(A)$ such that $\C_V(R)\neq 0$. Since $Z$ acts as the multiplication by $-1$ on $V$, we get $R\cap Z=1$ and $\Omega_1(A)=Z\times R$. Therefore, replacing $A_0$ by $R$ if necessary, we may assume that $R=A_0$. By~\cite[\S~$1$, page~$7$]{Sz}, we have $V=\C_V(A_0)\times [V,A_0]$. By construction, $\C_V(A_0)\neq 0$, $A_0$ acts faithfully on $[V,A_0]$ and the kernel of the action of $A$ on $\C_V(A_0)$ is $A_0$. It follows by Lemma~\ref{PabloLemma} that $|A_0|\leq 2^{r}$ (respectively $2^{r-1}$ if $A\leq S$) with $r=\dim_{\mathbb{F}_q}[V,A_0]$. Since $A/A_0\cong Q_8$ is non-abelian, $s=\dim_{\mathbb{F}_q}\C_V(A_0)\geq 2$. Hence
$$|\overline{A}|=4|A_0|\leq 2^{s}\cdot 2^{r}\leq 2^{r+s}=2^n\quad (\textrm{respectively }2^{s}\cdot 2^{r-1}=2^{n-1} \textrm{ if }A\leq S)$$
and the lemma is proved.
\end{proof}

We now apply Lemma~\ref{e:PSL} to obtain upper bounds for $e_T$ when $T$ is a simple group of Lie type, which we report in Table~\ref{lie}. 

When $T$ has odd characteristic, this bound in obtained by using~\cite[Table~$5.4.C$, page~$200$]{KL}, which lists the minimum degree of a projective representation of every simple group of Lie type, and then applying Lemma~\ref{e:PSL}. For instance, we have that $G_2(q)$ has a projective representation of degree $7$, that is, $G_2(q)\leq \PSL(7,q)$. Hence, by Lemma~\ref{e:PSL}, we get $e_{G_2(q)}\leq e_{\PSL(7,q)}=2^6$. All the entries in the second column of Table~\ref{lie} are computed with this method. 

In the case of groups of even characteristic (except for $^2A_2$ and $^2B_2$), the bound is obtained by collecting classical and difficult results about the maximal order of unipotent abelian subgroups of $T$. Note that these groups are not necessarily elementary abelian. For example, when $q$ is even, the maximal order of a unipotent abelian subgroup of $E_6(q)$ is $q^{16}$ and a reference for this result is~\cite{V}. Therefore $e_{E_6(q)}\leq q^{16}$. We stress that we do not claim that $q^{16}=e_{E_6(q)}$.

Finally, Table~\ref{lie} gives the exact value of $e_{{^2\!A_2}(q)}$ and $e_{{^2\!B_2}(q)}$, which can be extracted from~\cite{S} and~\cite{W}. We are now ready to prove Theorem~\ref{theo:key}. 

\begin{table}[!h]
\begin{center}
\begin{tabular}{|c|c|c|c|c}\hline
Group&$q$ odd & $q$ even &Reference\\\hline
$A_{2n}$&$2^{2n}$&$q^{n(n+1)}$&\cite{B}\\
$A_{2n+1}$&$2^{2n+2}$&$q^{(n+1)^2}$&\cite{B}\\
$B_n,n\geq 2$&$2^{2n}$&$q^{n(n+1)/2}$&\cite{B}\\
$C_{n},n\geq 3$&$2^{2n}$&$q^{n(n+1)/2}$&\cite{B}\\
$D_{n},n\geq 4$&$2^{2n}$&$q^{n(n-1)/2}$&\cite{B}\\
$E_{6}$&$2^{26}$&$q^{16}$&\cite{V}\\
$E_{7}$&$2^{56}$&$q^{27}$&\cite{V}\\
$E_{8}$&$2^{248}$&$q^{36}$&\cite{V}\\
$F_4$&$2^{26}$&$q^{11}$&\cite{V}\\
$G_{2}$&$2^{6}$&$q^3$&\cite{V}\\
$^2A_2$&$2^2$&$q$&\cite{W}\\
$^2A_{2n}$&$2^{2n}$&$q^{n^2+1}$&\cite{W}\\
$^2A_{2n+1},n\geq 1$&$2^{2n+2}$&$q^{(n+1)^2}$&\cite{W}\\
$^2B_{2}$&-&$q$&\cite{S}\\
$^2D_{n},n\geq 5$&$2^{2n}$&$q^{(n-1)(n-2)/2+2}$&\cite{W2}\\
$^2D_4$&$2^8$&$q^6$&\cite{W2}\\
$^3D_{4}$&$2^{8}$&$q^5$&\cite{V}\\
$^2E_{6}$&$2^{26}$&$q^{12}$&\cite{V}\\
$^2F_{4}$&-&$q^5$&\cite{V}\\
$^2G_{2}$&$2^{3}$&-&\cite{R}\\\hline
\end{tabular}
\end{center}
\caption{Upper bound for $e_T$ for groups of Lie type}\label{lie}
\end{table}

\medskip

\noindent\emph{Proof of Theorem~\ref{theo:key}. }

Let $T,l,t,o,l_e,l_o$ and $e$ be as in the statement of Theorem~\ref{theo:key}. Given $e$ and $o$, write $(\dag)_l$ for the inequality $l_oo^l>6le^{3l/2}\log_2(e)$ in the variable $l$. We claim that if $(\dag)_1$ holds (that is, $o>6e^{3/2}\log_2(e)$), then $(\dag)_{l}$ holds for every $l\geq 1$. Indeed,
\begin{eqnarray*}
l_oo^l&> &l_o(6e^{3/2}\log_2(e))^l\geq 6(6^{l-1}e^{3l/2}\log_2(e)^l)\geq 6le^{3l/2}\log_2(e)
\end{eqnarray*}
where in the last inequality we used $6^{l-1}\geq l$. With a similar computation, it is easy to show that if $(\dag)_2$ holds, then $(\dag)_l$ holds for every $l\geq 2$. In particular, in order to show that $(\dag)_l$ holds for every $l\geq 1$ (respectively $l\geq 2$), it suffices to prove that $(\dag)_1$ (respectively $(\dag)_2$) holds.

We divide the proof in different cases, depending on the isomorphism class of the non-abelian simple group $T$.

\smallskip

\noindent\textsc{Case $T$ is a sporadic simple group. }As $|\Aut(T):T|\leq 2$ for every sporadic simple group $T$, we get $e\leq 2^{\varepsilon}e_T$ with $\varepsilon=1$ if $|\Out(T)|=2$ and $\varepsilon=0$ otherwise.  Using~\cite[Table~1, page~viii]{ATLAS} and Table~\ref{spor}, it is immediate to check that, if $T\neq M_{12}$ and $M_{22}$, then  $(\dag)_1$ holds. It remains to consider the case that $T=M_{12}$ or $M_{22}$. If $T=M_{12}$, then, with \texttt{Magma}, we see that $e=16$ and $(\dag)_2$ holds. Similarly, if $T=M_{22}$, then, with \texttt{Magma}, we see that $e=32$ and $(\dag)_2$ holds.

\smallskip

\noindent\textsc{Case $T=\Alt(n)$. } If $n=5$, then $\Aut(T)=\Sym(5)$, $o=15$, $e=4$ and $6le^{3l/2}\log_2(e)=12l8^l$. It is easy to show that $l_o15^l>12l8^l$ for $l\geq 5$. Therefore $(\dag)_l$ holds for $l\geq 5$. If $n=6$, then $\Aut(T)=\PGammaL(2,9)$, $o=45$, $e=8$ and $6le^{3l/2}\log_2(e)=18l8^{3l/2}$.  It is easy to show that $l_o45^l>18l8^{3l/2}$ for $l\geq 5$. Therefore $(\dag)_l$ holds for $l\geq 5$. If $n=7$, then $o=315$, $e=8$ and $o^2>12e^3\log_2(e)$, hence $(\dag)_l$ holds for every $l\geq 2$. If $n=8$, then $o=315$, $e=16$ and $l_oo^l>6le^{3l/2}\log_2(e)$ for $l\geq 3$. It follows that $(\dag)_l$ holds for $l\geq 3$. From now on, we assume that $n\geq 9$. In particular, we have $\Aut(T)=\Sym(n)$ and, by Lemma~\ref{e:alt}, $e=2^{\lfloor n/2\rfloor}$. It follows that $6e^{3/2}\log_2(e)\leq 6\lfloor n/2\rfloor 2^{3n/4}$. It is immediate to check that $o>6\lfloor n/2\rfloor 2^{3n/4}$ and hence $(\dag)_1$ holds.

\smallskip

It remains to deal with the case of groups of Lie type. We follow~\cite{ATLAS} for notation and terminology, although we sometimes write $\mathrm{PSL}(n+1,q)$ instead of  $A_n(q)$ when we need to emphasise some elementary property of the projective special linear group. Let $T$ be a group of Lie type over the base field of order $q=p^f$, where $p$ is a prime. We refer to~\cite[Table~5, page~xvi]{ATLAS} for information about $|T|$ and $|\Out(T)|$. The outer automorphism group of $T$ is the semidirect product (in this order) of groups of order $d$ (the \emph{diagonal} automorphisms), $f$ (the \emph{field} automorphisms) and $g$ (the \emph{graph} automorphisms of the corresponding Dynkin diagram), except when $T$ is one of $B_2(2^f)$, $G_2(3^f)$ or $F_4(2^f)$, in which case the \emph{extraordinary} graph automorphism squares to a generator of the field automorphisms. The groups of order $d$, $f$ and $g$ are cyclic, except when $T=D_4(q)$, in which case the group of graph automorphisms is $\Sym(3)$. We use these facts later.

Write $\varepsilon_d=1$ if $d$ is even and $\varepsilon_d=0$ if $d$ is odd. Similarly, write $\varepsilon_f=1$ if $f$ is even and $\varepsilon_f=0$ if $f$ is odd. In the sequel, we will make use of the upper bounds for $e_T$ appearing in Table~\ref{lie}. The lower bounds for $|T|$ are obtained by using the inequality $q^i-1\geq q^{i-1}$ for $i\geq 1$. For instance, $|A_1(q)|=(q+1)q(q-1)/(2,q-1)\geq q^2$.

\smallskip

\noindent\textsc{Case $T=A_1(q)$, $q=p^f$. }We have $|T|=q(q^2-1)/d$ with $d=(2,q-1)$. Also, $|\Out(T)|=df$.

\noindent\textsc{Subcase $p=2$. }As $A_1(2)$ is soluble and $A_1(4)=\Alt(5)$, we may assume that $f\geq 3$. Clearly, $d=1$. An elementary abelian subgroup of $\Aut(T)=\PGammaL(2,q)$ has order at most $2^{f+\varepsilon_f}$. As a Sylow 2-subgroup of $T$ is
elementary abelian and has order $2^f$, we have $2^f\leq e$ and hence $2^f\leq e\leq 2^{f+\varepsilon_f}$. Now, we show that $e=2^f$. If $f$ is odd, then $\varepsilon_f=0$ and hence there is nothing to prove. Assume that $f$ is even. We argue by contradiction and we assume that $\PGammaL(2,q)$ contains an elementary abelian 2-subgroup $E$ of order $2^{f+1}$. As $\Out(T)$ is cyclic, the group $Q=E\cap T$ has order $2^f$ and hence is a Sylow 2-subgroup of $T$. Since $Q$ is self-centralising in $\PGammaL(2,q)$, we get $E\leq \
C_{\PGammaL(2,q)}(Q)=Q$, which is a contradiction.

From the previous paragraph, we have $e=q$ and $6e^{3/2}\log_2(e)= 6q^{3/2}f$. It is easy to check that $l_o(q^2-1)^l>6lq^{3l/2}f$ if and only if  $f\geq 13$ and $l\geq 1$, or $f\geq 7$ and $l=2$, or $f\geq 3$ and $l= 3$, or $l=4$ and $f\geq 4$, or $l\geq 5$.

\noindent\textsc{Subcase $p>2$. }Clearly, $d=2$. As $A_1(3)$ is soluble, $A_1(5)=A_1(4)$, $A_1(7)=A_2(2)$ (which we shall study later) and $A_1(9)=\Alt(6)$, we may assume that $q\neq 3,5,7,9$. An elementary abelian subgroup of $\Aut(T)=\PGammaL(2,q)$ has order at most $2^{2+\varepsilon_f}$. Thence $6e^{3/2}\log_2(e)=48\cdot 2^{3\varepsilon_f/2}(2+\varepsilon_f)$. Also, as $(q+1,q-1)=2$, we have $o\geq q(q-1)/2$. For $(p,f)\neq (5,2),(11,1),(13,1)$, we have $q(q-1)/2>48\cdot 2^{3\varepsilon_f/2}(2+\varepsilon_f)$ and $(\dag)_1$ holds. If $T=A_1(11)$ or $A_1(13)$, we have $e=4$, $(q(q-1)/2)^2>12e^3\log_2(e)$ and hence $(\dag)_2$ holds. If $T=A_1(25)$, we have $e=8$, $(q(q-1)/2)^2>12e^3\log_2(e)$ and hence $(\dag)_2$ holds.

\smallskip

\noindent\textsc{Case $T=A_n(q)$, $q=p^f$, $n\geq 2$. }We have $|T|=q^{n(n+1)/2}\prod_{i=1}^n(q^{i+1}-1)/d$ with $d=(n+1,q-1)$. Also, $|\Out(T)|=df2$.

\noindent\textsc{Subcase $p=2$. }Clearly, $d$ is odd and hence $e_{\Out(T)}\leq 2^{1+\varepsilon_f}$.  Assume $n=2m$. Since an elementary abelian subgroup of $T$ has order at most $q^{m(m+1)}$, we get $e\leq 2^{1+\varepsilon_f}q^{m(m+1)}$. Using this inequality, for $(f,m)\neq (1,1),(1,2),(2,1),(3,1)$ and $(4,1)$, we have $o>6e^{3/2}\log_2(e)$ and $(\dag)_1$ holds.   For $(f,m)=(1,1)$, we have $T=A_2(2)$ and, with \texttt{Magma}, we see that $e=4$, $o=21$ and $l_oo^l>6le^{3l/2}\log_2(e)$ for every $l\neq  1,2,4$. 
For $(f,m)=(2,1)$, we have $T=A_2(4)$ and, with \texttt{Magma}, we see that $e=2^4$, $o=315$ and $l_oo^l>6le^{3l/2}\log_2(e)$ for every $l\geq 3$. 
For $(f,m)=(3,1)$, we have $T=A_2(8)$ and, with \texttt{Magma}, we see that $e=2^6$ and $o>6e^{3/2}\log_2(e)$, hence $(\dag)_1$ holds. For $(f,m)=(4,1)$, we have $T=A_2(16)$ and, with \texttt{Magma}, we see that $e=2^8$ and $o>6e^{3/2}\log_2(e)$, hence $(\dag)_1$ holds. For $(f,m)=(1,2)$, we have $T=A_4(2)$ and, with \texttt{Magma}, we see that $e=2^6$ and $o^2>12e^{3}\log_2(e)$, hence $(\dag)_2$ holds. 

Assume $n=2m+1$. As $A_3(2)=\Alt(8)$, we may assume that $(m,f)\neq (1,1)$. Since an elementary abelian subgroup of $T$ has order at most $q^{(m+1)^2}$ and $e_{\Out(T)}\leq 2^{1+\varepsilon_f}$, we get 
$e\leq 2^{1+\varepsilon_f}q^{(m+1)^2}$. Using this inequality, for $(f,m)\neq (1,2)$ and $(2,1)$, we have $o>6e^{3/2}\log_2(e)$ and $(\dag)_1$ holds.   For $(f,m)=(2,1)$, we have $T=A_3(4)$ and, with \texttt{Magma}, we see that $e=2^8$ and  $o>6e^{3/2}\log_2(e)$, hence $(\dag)_1$ holds. 
For $(f,m)=(1,2)$, we have $T=A_5(2)$ and, with \texttt{Magma}, we see that $e=2^9$ and $o^2>12e^{3}\log_2(e)$, hence $(\dag)_2$  holds. 

\noindent\textsc{Subcase $p>2$. }From Lemma~\ref{e:PSL}, we obtain $e_{\mathrm{PGL}(n+1,q)}=2^{n+1}$. Furthermore, $|\Aut(T):\mathrm{PGL}(n+1,q)|=2f$. Hence $e\leq 2^{n+2+\varepsilon_f}$. Using this inequality, it is easy to check that, for $(q,n)\neq (3,2)$, we have  $o>6e^{3/2}\log_2(e)$ and hence $(\dag)_1$ holds. For $(q,n)=(3,2)$, we have $T=A_2(3)$ and, with \texttt{Magma}, we see that $e=8$, $o=351$ and $o^2>12e^{3}\log_2(e)$, hence $(\dag)_2$ holds.

\smallskip

\noindent\textsc{Case $T=B_2(q)$, $q=p^f$. }We have $|T|=q^4(q^4-1)(q^2-1)/d$ with $d=(2,q-1)$. Also, $|\Out(T)|=df2$ if $p=2$ (in which case the subgroup of $\Out(T)$ corresponding to $f2$ is cyclic, see~\cite[page~xv]{ATLAS}) and $|\Out(T)|=df$ if $p\neq 2$.

\noindent\textsc{Subcase $p=2$. }Clearly, $d=1$. Since $B_2(2)=\Sym(6)$, we may assume that $f\geq 2$. As an elementary abelian $2$-subgroup of $T$ has order at most $q^3$, we have $e\leq 2q^{3}$. Using this inequality, for every $f\geq 6$, we get $(q^4-1)(q^2-1)>6e^{3/2}\log_2(e)$ and hence $(\dag)_1$ holds. Also, using again $e\leq 2q^3$, for $f=4,5$, we obtain $((q^4-1)(q^2-1))^2>12e^3\log_2(e)$ and hence $(\dag)_2$ holds. If $f=3$,  then, with \texttt{Magma}, we see that $e=2^9$. Now, with a direct computation, we see that  $o^2>12e^{3}\log_2(e)$ and hence $(\dag)_2$ holds. Finally, if $f=2$, then, with \texttt{Magma}, we see that $e=2^6$. Using this value for $e$, it is easy to check with a direct computation that $(\dag)_l$ holds for every $l\geq 3$.

\noindent\textsc{Subcase $p>2$. }Clearly, $d=2$. Also, as $(q^2-1,q^2+1)=2$, we have $o>q^4(q^2-1)/2$. Since $B_2(q).d=\mathrm{PSp}(4,q).d\leq \mathrm{PGL}(4,q)$, from Lemma~\ref{e:PSL} we obtain $e_{B_2(q).d}=2^4=16$. Therefore, $e\leq 2^{4+\varepsilon_f}$. Using this inequality, for $q\geq 5$, we have that   $o\geq 5^4(5^2-1)/2>6e^{3/2}\log_2(e)$ and hence $(\dag)_1$ holds. Assume $q=3$. With \texttt{Magma}, we see that $e=16$, $o=405$ and $(\dag)_l$ holds for every $l\geq 3$. 

\smallskip

\noindent\textsc{Case $T=B_n(q)$, $q=p^f$, $n\geq 3$. }We have $|T|=q^{n^2}\prod_{i=1}^n(q^{2i}-1)/d$ with $d=(2,q-1)$. Also, $|\Out(T)|=df$. 

\noindent\textsc{Subcase $p=2$. }Clearly, $d=1$ and $o=\prod_{i=1}^n(q^{2i}-1)$. As an elementary abelian $2$-subgroup of $T$ has order at most $q^{n(n+1)/2}$, we have $e\leq 2^{\varepsilon_f}q^{n(n+1)/2}$. Using this inequality, it is easy to verify that for $n\geq 5$, we have $o>6e^{3/2}\log_2(e)$ and hence $(\dag)_1$ holds.  

For the remaining values of $n$ (that is, $n\in\{3,4\}$), we get that $o>6e^{3/2}\log_2(e)$ if and only if $n=4$ and $f>1$, or $n=3$ and $f>2$. In particular, it remains to study the groups $B_3(2),B_3(4)$ and $B_4(2)$. If $T=B_3(2)$, then, with \texttt{Magma}, we see that $e=2^6$, $o=2835$ and $l_oo^l>6le^{3l/2}\log_2(e)$ for every $l\geq 3$. If $T=B_3(4)$, then, with \texttt{Magma}, we see that $e=4^6$, $o=15663375$ and $o^2>12e^{3}\log_2(e)$, hence $(\dag)_2$ holds. Similarly, if $T=B_4(2)$, then, with \texttt{Magma}, we see that $e=2^{10}$ and $o^2>12e^{3}\log_2(e)$, hence $(\dag)_2$ holds. 

\noindent\textsc{Subcase $p>2$. }Clearly, $d=2$. Since $(q^n-1,q^n+1)=2$, we have $o\geq q^{n^2}(q^n-1)/2\geq 3^{n^2}(3^n-1)/2$. As an elementary abelian subgroup of $T$ has order at most $2^{2n}$ and $e_{\Out(T)}\leq 2^{1+\varepsilon_f}$, we have $e\leq 2^{2n+1+\varepsilon_f}\leq 2^{2n+2}$. Using this inequality, we get $o\geq 3^{n^2}(3^n-1)/2>6e^{3/2}\log_2(e)$ and $(\dag)_1$ holds.   

\smallskip

\noindent\textsc{Case $T=C_n(q)$, $q=p^f$. }We have $|T|=q^{n^2}\prod_{i=1}^n(q^{2i}-1)/d$ with $d=(2,q-1)$ and $n\geq 3$. Also, $|\Out(T)|=df$. 

\noindent\textsc{Subcase $p=2$. }We have $B_n(q)=C_n(q)$ and there is nothing to prove.

\noindent\textsc{Subcase $p>2$. }This subcase is exactly as the subcase $B_n(q)$ with $q$ odd.

\smallskip

\noindent\textsc{Case $T=D_n(q)$, $q=p^f$. }We have $|T|=q^{n(n-1)}(q^n-1)\prod_{i=1}^{n-1}(q^{2i}-1)/d$ with $d=(4,q^n-1)$ and $n\geq 4$. Also, $|\Out(T)|=df6$ if $n=4$ and $|\Out(T)|=df2$ if $n>4$. In particular, $e_{\Out(T)}=2^{1+\varepsilon_d+\varepsilon_f}$.

\noindent\textsc{Subcase $p=2$. }
Clearly, $d=1$ and $o\geq (q^{n}-1)\prod_{i=1}^{n-1}q^{2i-1}=(q^n-1)q^{(n-1)^2}$. As an elementary abelian $2$-subgroup of $T$ has order at most $q^{n(n-1)/2}=2^{n(n-1)f/2}$, we have $e\leq 2^{n(n-1)f/2+1+\varepsilon_f}$. It follows that $6e^{3/2}\log_2(e)\leq 6\cdot 2^{3/2+3\varepsilon_f/2}(n(n-1)f/2+1+\varepsilon_f)q^{3n(n-1)/4}$. Now, it is easy to verify that for $n\geq 6$, we have $(q^n-1)q^{(n-1)^2}>6\cdot 2^{3/2+3\varepsilon_f/2}(n(n-1)f/2+1+\varepsilon_f)q^{3n(n-1)/4}$ and hence $(\dag)_1$ holds.  

For the remaining values of $n$ (that is, $n\in\{4,5\}$), using the explicit formula for $|T|$ we get that $o>6\cdot 2^{3/2+3\varepsilon_f/2}(n(n-1)f/2+1+\varepsilon_f)q^{3n(n-1)/4}$ if and only if $n=5$, or $n=4$ and $f>1$. In particular, it remains to study the group $T=D_4(2)$. Using \texttt{Magma}, we see that $e=2^7$, that $(\dag)_1$ fails and that $o^2>12e^3\log_2(e)$. It follows that $(\dag)_l$ holds for $l\geq 2$. 

\noindent\textsc{Subcase $p>2$. }Clearly, $d$ is even and $o>q^{n(n-1)}$. As an elementary abelian $2$-subgroup of $T$ has order at most $2^{2n}$, we have  $e\leq 2^{2+\varepsilon_f}\cdot 2^{2n}=2^{2n+2+\varepsilon_f}$. Using this inequality it is easy to see that  for $q\geq 5$, we get $o> 5^{n(n-1)}>6e^{3/2}\log_2(e)$ and $(\dag)_1$ holds. Finally, for $q=3$ using that $o\geq 5\cdot 3^{n(n-1)}$, we obtain $o>6e^{3/2}\log_2(e)$ and hence $(\dag)_1$ holds.

\smallskip

\noindent\textsc{Case $T=E_6(q)$, $q=p^f$. }We have $|T|=q^{36}(q^{12}-1)(q^9-1)(q^8-1)(q^6-1)(q^5-1)(q^2-1)/d$ with $d=(3,q-1)$. Also, $|\Out(T)|=df2$. 

\noindent\textsc{Subcase $p=2$. }
Clearly, $d$ is odd and $o>q^{36}$.  As an elementary abelian $2$-subgroup of $T$ has order at most $q^{16}=2^{16f}$, we get $e\leq 2^{1+\varepsilon_f}\cdot q^{16}=2^{16f+1+\varepsilon_f}$. It follows that $6e^{3/2}\log_2(e)\leq 6\cdot 2^{3/2+3\varepsilon_f/2}q^{24}(16f+1+\varepsilon_f)$. Now, $o>q^{36}>  6 \cdot 2^{3/2+3\varepsilon_f/2}q^{24}(16f+1+\varepsilon_f)$ and hence $(\dag)_1$ holds.

\noindent\textsc{Subcase $p>2$. }Clearly, $d$ is odd and $o\geq q^{36}$. As an elementary abelian $2$-subgroup of $T$ has order at most $2^{26}$, we get  $e\leq 2^{1+\varepsilon_f}\cdot 2^{26}\leq 2^{28}$. Now, $o\geq q^{36}\geq 3^{36}>6e^{3/2}\log_2(e)$ and hence  $(\dag)_1$ holds.   

\smallskip

\noindent\textsc{Case $T=E_7(q)$, $q=p^f$. }We have $|T|=q^{63}(q^{18}-1)(q^{14}-1)(q^{12}-1)(q^{10}-1)(q^8-1)(q^6-1)(q^2-1)/d$ with $d=(2,q-1)$. Also, $|\Out(T)|=df$. 

\noindent\textsc{Subcase $p=2$. }
Clearly, $d=1$ and $o>q^{63}$. An  elementary abelian $2$-subgroup of $T$ has order at most  $q^{27}=2^{27f}$. Therefore $e\leq 2^{\varepsilon_f} q^{27}=2^{27f+\varepsilon_f}$. It follows that $6e^{3/2}\log_2(e)\leq 6 (2q^{27})^{3/2}(27f+1)$. Using this inequality, it is easy to check that $q^{63}>6e^{3/2}\log_2(e)$ and hence $(\dag)_1$ holds.  

\noindent\textsc{Subcase $p>2$. }Clearly, $d=2$ and $o\geq q^{63}$. As an elementary abelian $2$-subgroup of $T$ has order at most $2^{56}$, we get  $e\leq 2^{1+\varepsilon_f}\cdot 2^{56}\leq 2^{58}$. Now, $o\geq q^{63}\geq 3^{63}>6e^{3/2}\log_2(e)$ and hence $(\dag)_1$ holds.   

\smallskip

\noindent\textsc{Case $T=E_8(q)$, $q=p^f$. }We have 
$|T|=q^{120}(q^{30}-1)(q^{24}-1)(q^{20}-1)(q^{18}-1)(q^{14}-1)(q^{12}-1)(q^8-1)(q^2-1)$. Also, $|\Out(T)|=f$.

\noindent\textsc{Subcase $p=2$. }
Clearly, $o>q^{120}$. An elementary abelian $2$-subgroup of $T$ has order at most $q^{36}=2^{36f}$. Therefore $e\leq 2^{\varepsilon_f}\cdot q^{36}=2^{36f+\varepsilon_f}$. It follows that $6e^{3/2}\log_2(e)\leq 6\cdot 2^{3\varepsilon_f/2}q^{54}(36f+\varepsilon_f)$. Now, $o>q^{120}> 6\cdot 2^{3\varepsilon_f/2}q^{54}(36f+\varepsilon_f)$ and hence $(\dag)_1$ holds.  

\noindent\textsc{Subcase $p>2$. }Note that  $(q^{4n+2}-1)=(q^2-1)(q^{4n}+q^{4n-2}+\cdots +q^2+1)$ and $q^{4n}+q^{4n-2}+\cdots +q^2+1$ is odd (because it is the sum of $(2n+1)$ odd summands). Hence the higher power of $2$ dividing $q^{4n+2}-1$ is at most $q^2-1$. Using the formula for $|T|$ and this remark, we obtain that a Sylow $2$-subgroup of $T$ has order at most $(q^2-1)^4(q^{24}-1)(q^{20}-1)(q^{12}-1)(q^8-1)<q^{2\cdot 4+24+20+12+8}=q^{72}$.  It follows that $6e^{3/2}\log_2(e)\leq 6(2q^{72})^{3/2}\log_2(2q^{72})$. Using this inequality, we obtain $o>q^{120}>6e^{3/2}\log_2(e)$ and hence $(\dag)_1$ holds.   

\smallskip

\noindent\textsc{Case $T=F_4(q)$, $q=p^f$. }We have $|T|=q^{24}(q^{12}-1)(q^8-1)(q^6-1)(q^2-1)$. Also, $|\Out(T)|=f2$ if $p=2$ (in which case $\Out(T)$ is cyclic, see~\cite[page~xv]{ATLAS}) and $|\Out(T)|=f$ if $p>2$.

\noindent\textsc{Subcase $p=2$. }
Clearly, $o>q^{24}$. As an elementary abelian $2$-subgroup of $T$ has order at most $q^{11}=2^{11f}$ and $\Out(T)$ is cyclic, we get $e\leq 2q^{11}=2^{11f+1}$. Using this inequality, for $f\geq 2$ we obtain $o>q^{24}>6e^{3/2}\log_2(e)$ and hence $(\dag)_1$ holds. If $f=1$, then using the explicit value for $o$ we also obtain   $o>6e^{3/2}\log_2(e)$ and hence $(\dag)_1$ holds.

\noindent\textsc{Subcase $p>2$. } As an elementary abelian $2$-subgroup of $T$ has order at most $2^{26}$, we get $e \leq 2^{\varepsilon_f}\cdot 2^{26}\leq 2^{27}$. For $q\geq 5$, we have $o\geq 5^{24}>6\cdot (2^{27})^{3/2}\cdot 27$ and hence $(\dag)_1$ holds. Finally, if $q=3$, then using the explicit value of $o$ we also obtain $o>6\cdot(2^{27})^{3/2}\cdot 27$ and hence $(\dag)_1$ holds.

\smallskip

\noindent\textsc{Case $T=G_2(q)$, $q=p^f$. }We have $|T|=q^6(q^6-1)(q^2-1)$. Also, $|\Out(T)|=f$ if $p\neq 3$ and $|\Out(T)|=f2$ if $p=3$. 

\noindent\textsc{Subcase $p=2$. }As $G_2(2)$ is not simple and $G_2(2)'= {^2A_2(3)}$ (which we shall study later), we may assume $f\geq 2$. As an elementary abelian $2$-subgroup of $T$ has order at most $q^{3}=2^{3f}$, we get $e\leq 2^{\varepsilon_f}\cdot q^{3}=2^{3f+\varepsilon_f}$. It follows that $6e^{3/2}\log_2(e)\leq 6\cdot 2^{3\varepsilon_f/2}\cdot q^{9/2}\cdot (3f+\varepsilon_f)$. Now, $o=(q^6-1)(q^2-1)> 6\cdot 2^{3\varepsilon_f/2}\cdot q^{9/2}\cdot (3f+\varepsilon_f)
$ and hence $(\dag)_1$ holds.  

\noindent\textsc{Subcase $p>2$. }Assume $q>3$. As an elementary abelian $2$-subgroup of $T$ has order at most $2^{6}$, we get $e\leq 2^{1+\varepsilon_f}\cdot 2^{6}\leq 2^8$. Since $(q^3-1,q^3+1)=2$, we have $o\geq q^6(q^3-1)/2$. Now, $o\geq 5^6(5^3-1)/2>6\cdot (2^8)^{3/2}\cdot 8$ and hence $(\dag)_1$ holds. Finally, assume $q=3$. Using \texttt{Magma}, we see that  $e=16$, $o=66339$ and $o>6e^{3/2}\log_2(e)$ and hence $(\dag)_1$ holds.

\smallskip

\noindent\textsc{Case $T={^2A}_n(q)$, $q=p^f$. }We have $|T|=q^{n(n+1)/2}\prod_{i=1}^n(q^{i+1}-(-1)^{i+1})/d$ with $n\geq 2$ and $d=(n+1,q+1)$. Also, $|\Out(T)|=df2$ and the subgroup of $\Out(T)$ corresponding to $f2$ is cyclic (being the Galois group of the defining field for the unitary group $T$). Therefore $e_{\Out(T)}\leq 2^{1+\varepsilon_d}$.  Recall that  ${^2A}_2(2)$ is soluble and ${^2A}_3(2)=B_2(3)$. 

\noindent\textsc{Subcase $p=2$. }Clearly $d$ is odd. Assume $n=2m+1$. As an elementary abelian $2$-subgroup of $T$ has order at most $q^{(m+1)^2}$, we get $e\leq 2q^{(m+1)^2}$. Using  this inequality, for $(m,f)\neq (1,1),(1,2)$ and $(2,1)$, we have $o>6e^{3/2}\log_2(e)$ and hence $(\dag)_1$ holds. Also, as ${^2A}_3(2)\cong B_2(3)$ and we have already studied $B_2(3)$, we may assume that $(m,f)\neq (1,1)$. Assume $T={^2A}_3(4)$. With \texttt{Magma}, we see that $e=2^{8}$. As $o>6e^{3/2}\log_2(e)$, we obtain that $(\dag)_1$ holds. Assume $T={^2A}_5(2)$. With \texttt{Magma}, we see that $e=2^{9}$. As $o^2>12e^3\log_2(e)$, we obtain that $(\dag)_2$ holds.

Now, assume $n=2$. As an elementary abelian $2$-subgroup of $T$ has order at most $q$, we get $e\leq 2q=2^{1+f}$. Using  this inequality,  we have $o>6e^{3/2}\log_2(e)$ and $(\dag)_1$ holds. 

Finally, assume $n=2m$ with $m>1$. As an elementary abelian $2$-subgroup of $T$ has order at most $q^{m^2+1}$, we get $e\leq 2q^{m^2+1}=2^{1+f+m^2f}$. Using  this inequality, for $(m,f)\neq (2,1)$ we have $o>6e^{3/2}\log_2(e)$ and $(\dag)_1$ holds. If $T={^2A}_4(2)$, then we see, with \texttt{Magma}, that $e=16$ and $o>6e^{3/2}\log_2(e)$ and hence $(\dag)_1$ holds.

\noindent\textsc{Subcase $p>2$. }As an elementary abelian $2$-subgroup of $T$ has order at most $2^{n+1}$, we get $e\leq 2^{n+2+\varepsilon_d}$ and $6e^{3/2}\log_2(e)\leq 3\cdot 2^{3n/2+4+3\varepsilon_d/2}(n+2+\varepsilon_d)$. With this inequality, we have that, for $(n,q)\neq (2,3)$, the inequality $(\dag)_1$ holds. Assume $T={^2A}_2(3)$. Clearly, $d=1$ and $o=189$. With \texttt{Magma}, we see that $e=8$. Now, $o^2>12e^3\log_2(e)$ and hence $(\dag)_2$ holds.

\smallskip

\noindent\textsc{Case $T={^2}B_2(q)$, $q=2^{2m+1}$. }We have 
$|T|=q^{2}(q^2+1)(q-1)$. Since ${^2B_2(2)}$ is soluble, we may assume that $m\geq 1$. We have $|\Out(T)|=2m+1$ and hence an elementary abelian $2$-subgroup of $\Aut(T)$ is contained in $T$.  A maximal elementary abelian $2$-subgroup of $T$ has order $q$. Therefore $e= q$ and $6e^{3/2}\log_2(e)= 6q^{3/2}(2m+1)$.  It is easy to check that $o=(q^2+1)(q-1)>6q^{3/2}(2m+1)$ and hence $(\dag)_1$ holds.

\smallskip

\noindent\textsc{Case $T={^2}D_n(q)$, $q=p^f$. }We have $n\geq 4$ and $|T|=q^{n(n-1)}(q^{n}+1)\prod_{i=1}^{n-1}(q^{2i}-1)/d$ with $d=(4,q^n+1)$. Also, $|\Out(T)|=df2$.

\noindent\textsc{Subcase $p=2$. }Clearly, $d=1$. Assume $n=4$. An elementary abelian $2$-subgroup of $T$ has order at most $q^6$. Hence $e\leq 2^{1+\varepsilon_f}q^6$ and $6e^{3/2}\log_2(e)\leq 6\cdot 2^{3/2+3\varepsilon_f/2}q^{9}(6f+1+\varepsilon_f)$. With this inequality, it is easy to check that $(\dag)_1$ holds for $f\geq 2$. If $f=1$, then $o=48195$, $e\leq 2^7$ and $o^2>12e^{3}\log_2(e)$ and hence $(\dag)_2$ holds.

Assume $n\geq 5$. An elementary abelian $2$-subgroup of $T$ has order at most $q^{(n-1)(n-2)/2+2}$. Hence $e\leq 2^{1+\varepsilon_f}q^{(n-1)(n-2)/2+2}$. With this inequality, it is easy to check that $o>q^{(n-1)^2+n}>6e^{3/2}\log_2(e)$ and hence $(\dag)_1$ holds.

\noindent\textsc{Subcase $p>2$. }An elementary abelian $2$-subgroup of ${^2}D_n(q)$ has order at most $2^{2n}$ and hence $e\leq 2^{2n+1+\varepsilon_d+\varepsilon_f}\leq 2^{2n+3}$. With this inequality, it is easy to prove (for $(n,q)\neq (4,3)$) that $o>q^{n(n-1)}>6e^{3/2}\log_2(e)$ and hence $(\dag)_1$ holds. If $(n,q)=(4,3)$, we have $\varepsilon_f=0$ and $e\leq 2^{10}$. With a direct computation we see that  $o>6e^{3/2}\log_2(e)$ and hence $(\dag)_1$ holds.

\smallskip

\noindent\textsc{Case $T={^3D}_4(q)$, $q=p^f$. }We have $|T|=q^{12}(q^8+q^4+1)(q^6-1)(q^2-1)$ and $|\Out(T)|=3f$.

\noindent\textsc{Subcase $p=2$. }As an elementary abelian $2$-subgroup of $T$ has order at most $q^5$, we get $e\leq 2^{\varepsilon_f}q^5$ and $6e^{3/2}\log_2(e)\leq 6\cdot 2^{3\varepsilon_f/2}q^{15/2}(5f+\varepsilon_f)$. With this inequality it is easy to check that $o>q^{14}> 6\cdot 2^{3\varepsilon_f/2}q^{15/2}(5f+\varepsilon_f)$ and hence $(\dag)_1$ holds.

\noindent\textsc{Subcase $p>2$. }As an elementary abelian $2$-subgroup of $T$ has order at most $2^8$, we get $e\leq 2^{8+\varepsilon_f}\leq 2^9$. Now, $o\geq 3^{12}(3^8+3^4+1)>6\cdot 2^{27/2}\cdot 9\geq 6e^{3/2}\log_2(e)$ and hence $(\dag)_1$ holds.

\smallskip

\noindent\textsc{Case $T={^2}E_6(q)$, $q=p^{f}$. }We have 
$|T|=q^{36}(q^{12}-1)(q^9+1)(q^8-1)(q^6-1)(q^5+1)(q^2-1)/d$ with $d=(3,q+1)$. Also, $|\Out(T)|=df2$.

\noindent\textsc{Subcase $p=2$. }Clearly, $d$ is odd. A maximal elementary abelian $2$-subgroup of $T$ has order at most $q^{12}=2^{12f}$. Therefore $e\leq 2^{1+\varepsilon_f}q^{12}=2^{12f+1+\varepsilon_f}$. It follows that $6e^{3/2}\log_2(e)\leq 6\cdot 2^{3/2+3\varepsilon_f/2}\cdot q^{18}\cdot (12f+1+\varepsilon_f)$. With this inequality it is easy to check that $(\dag)_1$ holds.

\noindent\textsc{Subcase $p>2$. }Clearly, $d$ is odd. A maximal elementary abelian $2$-subgroup of $T$ has order at most $2^{26}$. Therefore $e\leq 2^{1+\varepsilon_f}2^{26}\leq 2^{28}$. It follows that $6e^{3/2}\log_2(e)\leq 6\cdot 2^{42}\cdot 28$. With this inequality it is easy to check that $o>3^{36}>6\cdot 2^{42}\cdot 28$ and hence $(\dag)_1$ holds.

\smallskip

\noindent\textsc{Case $T={^2}F_4(q)$, $q=2^{2m+1}$. }For $m\geq 1$, we have 
$|T|=q^{12}(q^6+1)(q^4-1)(q^3+1)(q-1)$. Also, $|\Out(T)|=f=2m+1$ and hence an elementary abelian $2$-subgroup of $\Aut(T)$ is contained in $T$ and has order at most $q^5$. Hence $e\leq q^5$. It follows that $6e^{3/2}\log_2(e)\leq 6q^{15/2}(10m+5)$. For $m\geq 1$, we get $o\geq q^{12}>6q^{15/2}(10m+5)$ and hence $(\dag)_1$ holds. If $m=0$, then $^2F_4(2)$ is not simple, the Tits group $T=$~$^2F_4(2)'$ is simple, $|^2F_4(2):{^2F_4}(2)'|=2$ and $^2F_4(2)=\Aut(^2F_4(2)')$. With \texttt{Magma}, we see that $e=32$, $o>6\cdot 32^{3/2}\cdot 5$ and hence $(\dag)_1$ holds.  

\smallskip

\noindent\textsc{Case $T={^2}G_2(q)$, $q=3^{2m+1}$. }We have 
$|T|=q^{3}(q^3+1)(q-1)$ and $|\Out(T)|=2m+1$. Since ${^2G_2(3)}$ is not simple and ${^2G_2(3)'=A_1(8)}$, we may assume that $m\geq 1$. As $|\Out(T)|$ is odd, an elementary abelian $2$-subgroup of $\Aut(T)$ is contained in $T$. Since a Sylow $2$-subgroup of $T$ has order $8$, we get $e\leq 8$. It follows that $6e^{3/2}\log_2(e)\leq 408$. Now, $o\geq |{^2G_2(27)}|/8>408$ and hence $(\dag)_1$ holds.~$\qed$

\section{Additional remarks}\label{OtherStuff}
\subsection{Cubic vertex-transitive graphs}\label{ss:cubic}
Tutte's theorem concerns the order of a vertex-stabiliser in a $3$-valent arc-transitive graph. Instead of trying to generalise it to valencies other than $3$, it is also possible to consider $3$-valent vertex-transitive graphs in general. It turns out that the problem of bounding the order of the vertex-stabiliser of a $3$-valent vertex-transitive graph is essentially equivalent to the problem of bounding it for 4-valent arc-transitive graphs. We now give a brief explanation of this possibly surprising fact. Let $(\Gamma,G)$ be a locally-$L$ pair such that $\Gamma$ has valency 3. If $L$ is transitive, then $\Gamma$ is $G$-arc-transitive and, by Tutte's theorem, $|G_v|\leq 48$. Similarly, if $L=1$, then $G_v=1$ because $\Gamma$ is connected. Since we are interested in graphs with `large' vertex-stabilisers, we ignore both of these cases. In particular, we may assume that $L\cong\C_2^{[3]}$, where $\C_2^{[3]}$ denotes the permutation group of order 2 and degree 3.

For each locally-$\C_2^{[3]}$ pair $(\Gamma,G)$, we construct an auxiliary locally-$\D_4$ pair $(\Merge(\Gamma),G)$. Conversely, for each locally-$\D_4$ pair $(\Gamma,G)$, we construct a locally-$\C_2^{[3]}$ pair $(\Split(\Gamma),G)$. Moreover, we will show that these constructions are inverses of each other.

\begin{definition}\label{def:a}
Let $(\Gamma,G)$ be locally-$\C_2^{[3]}$. As $\C_2^{[3]}$ fixes a unique point, each vertex $v\in\V\Gamma$ has a unique neighbour $v'\in\Gamma(u)$ with $G_v=G_{v'}$. Hence the set of pairs $\Sigma=\{\{v,v'\}:v\in\V\Gamma\}$ forms a system of imprimitivity for $G$. Let $\Merge(\Gamma)=\Gamma/\Sigma$.
\end{definition}

\begin{definition}\label{def:b}
Let $(\Gamma,G)$ be locally-$\D_4$. For every arc $a=(u,v)$ of $\Gamma$, there is a unique arc $a'=(u,w)$ such that $a$ and $a'$ have the same head and $G_a=G_{a'}$. Write $\overline{a}=\{a,a'\}$. We define a new graph $\Split(\Gamma)$ with vertices $\{\overline{a} : a\in\A\Gamma\}$ and two distinct elements $\overline{(u,v)}$ and $\overline{(w,x)}$ are adjacent if either $u=w$, or $u=x$ and $v=w$. Note that the set $\{\overline{a}:a\in\A\Gamma\}$ forms a system  of imprimitivity for $G$.
\end{definition}

\begin{lemma}\label{l:1}
Let $(\Gamma,G)$ be locally-$\C_2^{[3]}$ with $|G_v|\geq 4$. For a vertex $v\in\V\Gamma$, let $v'$ be the unique neighbour of $v$ with $G_v=G_{v'}$. Then $(\Merge(\Gamma),G)$ is locally-$\D_4$, $\Merge(\Gamma)$ has $|\V\Gamma|/2$ vertices and $|G_{\{v,v'\}}|=2|G_v|$. Moreover, $\Split(\Merge(\Gamma))\cong\Gamma$, with the isomorphism given by $\theta : \overline{(\{u,u'\},\{v,v'\})}\mapsto u$, where $u'\neq v\in\Gamma(u)$.
\end{lemma}
\begin{proof}
It is clear from Definition~\ref{def:a} that $G$ acts transitively on the arcs of $\Merge(\Gamma)$, which is connected and has $|\V\Gamma|/2$ vertices. For a vertex $v\in\V\Gamma$, let $K(v)$ denote the kernel of the action of $G_v$ on $\Gamma(v)$. Note that, if $u$ is adjacent to $v$ and $u\neq v'$, then $K(u)=K(v)$.

Suppose that $\Gamma$ contains a 3-cycle of the form $(u,v,v')$. Then $G_v$ fixes $v'$ and at least 2 neighbours of $v'$ hence $K(v)=K(v')$. If follows that $K(w)=K(v)$ for every neighbour $w\in\Gamma(v)$. Since $\Gamma$ is connected and $G$-vertex-transitive, we conclude that $K(v)=1$ and hence $|G_v|=2$, which is a contradiction. Now, suppose that $\Gamma$ contains a 4-cycle of the form $(u,u',v',v)$. Then, $G_v$ fixes $v$, $v'$, $u$ and, in particular, $u'$ hence $K(v)=K(v')$, which is a contradiction, for the same reasons as above. 

Since $\Gamma$ contains no such cycles, it is easily seen that $G$ acts faithfully on the vertices of $\Merge(\Gamma)$ and that $\Merge(\Gamma)$ is 4-valent. It follows that $|G_{\{v,v'\}}|=2|G_v|\geq 8$ and hence $(\Merge(\Gamma),G)$ must be locally-$\D_4$. The proof that $\theta$ is a well-defined isomorphism is straightforward.
\end{proof}

We also leave the proof of the next lemma to the reader.

\begin{lemma}\label{l:2}
Let $(\Gamma,G)$ be locally-$\D_4$. Then $(\Split(\Gamma),G)$ is locally-$\C_2^{[3]}$, $\Split(\Gamma)$ has $2|\V\Gamma|$ vertices and $|G_{\overline{a}}|=|G_v|/2$. Moreover, $\Merge(\Split(\Gamma))\cong\Gamma$, with the isomorphism given by $\{  \overline{(u,v)},\overline{(u,v)}' \}\mapsto u$.
\end{lemma}

Combining Lemmas~\ref{l:1} and~\ref{l:2} with Theorem~\ref{thm:main} yields Corollary~\ref{Cubic}.

\subsection{Normal subgroups of Djokovi\'{c}'s amalgams}\label{Amalgams}

In this section we will restate Theorem~\ref{thm:main} in a purely group theoretical language. Following Djokovi\'{c}~\cite{D},
we call a quintuple $(L,\varphi,B, \psi,R)$ an {\em amalgam} provided that $L, B$ and $R$ are finite groups and $\varphi: B \to L$, $\psi: B \to R$ are monomorphisms (the embeddings $\varphi$ and $\psi$ are often omitted from the notation when they are clear from the context).
Amalgams are usually given by means of an ambient group $G$ (called a {\em completion of the amalgam}), containing $L$ and $R$
as subgroups with $B=L\cap R$, and where $\varphi$ and $\psi$ are the inclusion mappings.
Note that, for each amalgam $(L,\varphi,B, \psi,R)$, there exists the {\em universal completion} $G^*$ (that is, the free product of $L$ and $R$ with amalgamation over $B$, and denoted by $L*_B R$),
with the property that every other completion $G$ is a quotient of $G^*$ by some normal subgroup $N$ intersecting both $L$ and $R$ trivially. We shall call such a quotient $G^* \to G^*/N\cong G$ a {\em smooth quotient}.
The {\em index} of the amalgam $(L,\varphi, B, \psi,R)$ is the pair $(|L:\varphi(B)|,|R:\psi(B)|)$. 
Finally, the amalgam is {\em faithful} if there is no nontrivial subgroup $H\le B$ with $\varphi(H)$ and $\psi(H)$ normal in $L$ and $R$, respectively.

Amalgams emerge naturally in many different contexts and areas of mathematics~\cite{serre}  and have a natural interpretation in the context of arc-transitive graphs. Namely, if $\Gamma$ is a finite $G$-arc-transitive graph of valency $k$, then 
$(G_v, G_{uv}, G_{\{u,v\}})$ (with the monomorphisms being the inclusion mappings) is a faithful amalgam of index $(k,2)$ and $G$ is a finite smooth quotient of the universal
group $G_v *_{G_{uv}} G_{\{u,v\}}$.
Conversely, given a finite smooth quotient $G\cong (L*_B R)/N$ of a faithful amalgam $(L,B,R)$ of index $(k,2)$, one can use the \emph{coset graph} construction to obtain a finite $G$-arc-transitive graph. Note that in this correspondence, the stabiliser $G_v$ corresponds to the group $L$ and  the permutation group $G_v^{\Gamma(v)}$ corresponds to the permutation group induced by the action of $L$ on the cosets of $B$ by right multiplication. If the latter permutation group is permutation isomorphic to  $P$, we say that the amalgam $(L,B,R)$ is of {\em local type $P$}.

The above gives
a natural correspondence between locally-$\D_4$ pairs and smooth completions of faithful amalgams of index $(4,2)$ and of local type $\D_4$.
Theorem~\ref{thm:main} can now be reformulated as follows.

\begin{theorem}
\label{the:amalgam}
Let $(L,B,R)$ be a faithful amalgam of index $(4,2)$ and local type $\D_4$. Let $m=|L|$ and let $N$ be a normal subgroup of $G^*=L*_B R$ of finite index $n$ which intersects both $L$ and $R$ trivially.
Then either
$$n \ge 2 m^2 \log_2(m/2)$$
or the corresponding coset graph $\Cos(G^*/N,L,a)$ is isomorphic either to $\C(r,s)$ for some $r\geq 3$, $1\leq s\leq \frac{r}{2}$ or to a graph from Tables~\ref{tb:soluble} and~\ref{tb:nsoluble}.
\end{theorem}

Faithful amalgams of local type $\D_4$ were completely determined by Djokovi\'c~\cite{D}. One of the consequences of his work is that in a faithful amalgam $(L,B,R)$ of local type $\D_4$,
the normaliser $N_L(B)$ of $B$ in $L$ has index $2$ in $L$ and is a nilpotent group of class at most $2$. Note that if the amalgam $(L,B,R)$ arises from the
locally-$\D_4$ pair $(\C(r,s), G)$, then $N_L(B)$ is elementary abelian.
This, together with Theorem~\ref{the:amalgam}, gives the following interesting consequence.

\begin{corollary}
\label{cor:amalgam}
Let $(L,B,R)$ be a faithful amalgam of index $(4,2)$ and local type $\D_4$ such that $N_L(B)$ is not elementary abelian. Let $|L|= m$ and let $N$ be a normal subgroup of $G^*=L*_B R$ of finite 
index $n$ which intersects both $L$ and $R$ trivially. Then either
 $$n \ge 2 m^2 \log_2(m/2)$$
or $\Cos(G^*/N,L,a)$ is one of the graphs in Tables~\ref{tb:soluble} and~\ref{tb:nsoluble}.
\end{corollary}

Finally let us mention an interesting result proved recently  by Meierfrankenfeld and Sami \cite{newsami}, which states that if $N$ and the amalgam $(L,B,R)$ are as in Corollary~\ref{cor:amalgam} and $n$ is odd, then $m\le 32$. Corollary~\ref{cor:amalgam} can therefore be viewed as a partial generalisation of the results in \cite{newsami}, where the condition on the index $n=|G^*:N|$
being odd is dropped, and the resulting upper bound on $m$ is of the form $o(\sqrt{n})$.

\bigskip
\footnotesize
\noindent\textit{Acknowledgments.} The second author is supported by UWA as part of the
Australian Council Federation Fellowship Project FF0776186.

\end{document}